\documentclass[oneside,11pt]{article}
\usepackage[margin=2cm]{geometry} 
\usepackage{amsthm, amsmath, amssymb}
\usepackage[mathlines]{lineno}
\usepackage[dvipsnames]{xcolor}
\usepackage[pagebackref,colorlinks,citecolor=Plum,urlcolor=Periwinkle,linkcolor=DarkOrchid]{hyperref}
\usepackage{graphicx}
\usepackage{enumerate}
\usepackage{nicefrac}
\usepackage[normalem]{ulem}

\def\ba#1\ea{\begin{linenomath}\begin{align*}#1\end{align*}\end{linenomath}}
\def\ban#1\ean{\begin{linenomath}\begin{align}#1\end{align}\end{linenomath}}

\numberwithin{equation}{section}
\newtheorem{theorem}{Theorem}[section]
\newtheorem{lemma}[theorem]{Lemma}
\newtheorem{corollary}[theorem]{Corollary}
\newtheorem{proposition}[theorem]{Proposition}
\newtheorem{conjecture}[theorem]{Conjecture}

\theoremstyle{definition}
\newtheorem{definition}[theorem]{Definition}
\newtheorem*{remark}{Remark}

\title{The trace-reinforced ants process does not find shortest paths}
\author{Daniel Kious\thanks{Department of Mathematical Sciences, University of Bath, Claverton Down, BA2 7AY Bath, UK.\newline Email: \texttt{d.kious/c.mailler@bath.ac.uk}} \thanks{DK is grateful to EPSRC for support through the grant EP/V00929X/1.} \and C\'ecile Mailler\footnotemark[1] \thanks{CM is grateful to EPSRC for support through the fellowship EP/R022186/1.} \and Bruno Schapira\thanks{Aix-Marseille Universit\'e, CNRS, Centrale Marseille, I2M, UMR 7373, 13453 Marseille, France.\newline Email: \texttt{bruno.schapira@univ-amu.fr}}}

\newcommand{\bs}{\boldsymbol}
\newcommand{\sss}{\ensuremath{\scriptscriptstyle}}

\begin{document}
\maketitle 

\begin{abstract}
In this paper, we study a probabilistic reinforcement-learning model
for ants searching for the shortest path(s) 
between their nest and a source of food.
In this model,
the nest and the source of food are two distinguished nodes $N$ and $F$ 
in a finite graph $\mathcal G$.
The ants perform a sequence of random walks on this graph, 
starting from the nest and stopped when first hitting the source of food.
At each step of its random walk, the $n$-th ant chooses to cross a neighbouring edge with probability proportional to the number of preceding ants that crossed that edge at least once.
We say that {\it the ants find the shortest path} if, almost surely as the number of ants grow to infinity, almost all the ants go from the nest to the source of food through one of the shortest paths, without {losing} time on other edges of the graph.

Our contribution is three-fold: (1) We prove that, if $\mathcal G$ is a tree rooted at $N$ whose leaves have been merged into node $F$, and with one edge between $N$ and $F$, then the ants  indeed find the shortest path. (2) In contrast, we provide three examples of graphs on which the ants do not find the shortest path, suggesting that in this model and in most graphs, ants do not find the shortest path. (3) In all these cases, we show that the sequence of normalised edge-weights converge to a {\it deterministic} limit, despite a linear-reinforcement mechanism, and we conjecture that this is a general fact which is valid on all finite graphs. To prove these results, we use stochastic approximation methods, and in particular the ODE method. One difficulty comes from the fact that this method relies on understanding the behaviour at large times of the solution of a non-linear, multi-dimensional ODE.
\end{abstract}

\section{Introduction and main results}
{\bf Context:}
It is believed that ants are able to find shortest paths between their nest and a source of food
with no other means of communication than the pheromones they lay behind them.
This phenomenon has been observed empirically in the biology literature 
(see, e.g.,~\cite{ants_89, Deneu90, current_ants}), and reinforcement-learning has been proposed as a model for {describing it} (see, e.g.~\cite{current_ants}).
In the survey of \cite[Chapter~1]{book_ant}, this phenomenon is called {\it stigmergy}: ``{\it ants stimulate other ants by modifying the environment via pheromone trail updating}''.

In this paper, we study a probabilistic reinforcement-learning model 
for this phenomenon of ants finding shortest path(s) between their nest and a source of food.
In this model, which was introduced in our previous paper~\cite{KMS},
the nest and the source of food are two nodes $N$ and $F$ of a finite graph $\mathcal G$,
and the ants perform successive random walks starting from~$N$ and stopped when first hitting~$F$.
The distribution of the $n$-th ant's walk depends on the trajectory of the previous $n-1$ random walks in a way that models ants leaving pheromones on the edges they cross:
each edge has a weight which is equal to~{1} at the start and which increases by~1
at time~$n$ if and only if the $n$-th ant has deposited pheromones on that edge.
The transition probabilities of the random walk of the $(n+1)$-th ant are proportional to the edge-weights at time~$n$.
The reinforcement is thus linear.

In~\cite{KMS}, the ants deposit pheromones only on their way back: 
i.e.\ when they come back to the nest after having hit~$F$.
Two cases are studied in~\cite{KMS}: 
\begin{itemize}
\item In the ``loop-erased'' ant process, or model (LE), ants come back to the nest following the loop-eras{ure} of the time-reversed version of their forward trajectory.
\item In the ``geodesic'' ant process, or model (G), they come back following the shortest path between $F$ and $N$ within the trace of their forward trajectory (ties are broken uniformly at random).
\end{itemize}
The conjecture is that, under these two versions of the model, when time goes to infinity, almost all ants go from $N$ to $F$ through a shortest path, i.e.\ {\it the ants indeed find the shortest path(s) between their nest and the food}. 
This conjecture is proved in~\cite{KMS} for model (LE) in the case when $\mathcal G$ is a series-parallel graph, and in the model (G) in the case when $\mathcal G$ is a five-edge graph called the lozenge.

\medskip
{\bf Main contribution:} In this paper, we look at the same model 
but assuming that ants deposit pheromones on their way forwards to the food, 
i.e.\ the weight of an edge increases by one at time $n$ if and only 
if the $n$-th ant has crossed this edge at least once on its way from $N$ to $F$. 
We call this model the ``trace-reinforced'' ant process, or model (T).
In the biology literature, all cases of ants depositing pheromones on their way forward{s}, 
backwards or both are considered (see~\cite[Chapter~1]{book_ant}).

Maybe surprisingly, this small change to the reinforcement rule leads to a drastically different behaviour:
indeed, we prove that, in the trace-reinforced ant process, in general,
{\it the ants do not find the shortest path(s) between their nest and the source of food}, 
except in some very particular cases. Indeed,
in Theorem~\ref{th:main}, we show that on a family of graphs called ``tree-like'', in which there is a unique edge between $N$ and $F$ ($N$ and $F$ are at distance~1), the ants do find the shortest path. However, {one can find graphs, with $N$ and $F$ at distance one, which are not tree-like and such that the ants do not find shortest paths} (see Proposition~\ref{prop:cornet}, Theorem~\ref{th:two_paths} and Proposition~\ref{prop:lozenge}).

The fact that ants do not always find shortest paths when depositing pheromones only on their way 
forward has been observed empirically on ants (see~\cite[Section~1.1.2]{book_ant}): 
``{\it The observation of real ant colonies has confirmed that ants that deposit pheromone only
when [going forward] to the nest are unable to find the shortest path between their nest and
the food source}''.
Our analysis proves that the model introduced in~\cite{KMS} exhibits the same behaviour.

Our second main finding, which might also be surprising at first glance, is that although we use a {\it linear} reinforcement mechanism (see the discussion below), in all the graphs we consider, the sequence of normalised edge-weights converges to a {\it deterministic} limit.
In fact, we conjecture that this is a general phenomenon, and that on all graphs with no multiple edges linked to node~$F$, the sequence of the normalised weights converges almost surely to a deterministic limit:
\begin{conjecture}\label{conj}
If, for all $n\geq 0$, for all $e\in E$, we let $W_e(n)$ denote the weight of edge~$e$ at time~$n$, 
then there exists a deterministic family $(\chi_e)_{e\in E}$ 
such that $(W_e(n)/n)_{e\in E}\to (\chi_e)_{e\in E}$ almost surely when $n\to+\infty$.
Furthermore, if in addition the distance between~$N$ and~$F$ is at least 2, then $\chi_e>0$ for all $e\in E$.
\end{conjecture}

\medskip
{\bf Discussion:}
Other {probabilistic reinforcement} models inspired by urns exist in the probability literature.
As far as we know these models are all self-reinforced random walks models with super-linear reinforcement; 
see for example Le Goff and Raimond~\cite{LGR} and Erhard, Franco and Reis~\cite{EFR}. 

The ant process can be seen as a ``path formation'' model: 
the quantity of interest is the subgraph of $\mathcal G$ 
obtained by removing from~$\mathcal G$ all edges whose normalised weight converges to zero.
If this limiting graph is different from~$\mathcal G$, then we say that ``some path(s) has formed''.

Another model for path formation is the P\'olya urns with graph based interactions
of Bena\"im, Benjamini, Chen and Lima~\cite{BBCL} 
and its generalisation to WARMs of~\cite{WARM}. 
In the latter (and in more recent paper on the same model such as Hirsch, Holmes and Kleptsyn~\cite{WARM_new}), 
only super-linear reinforcement is considered because it leads to path formation; 
in contrast, it is believed that, under linear reinforcement, 
the limiting graph would be equal to~$\mathcal G$.
{In~\cite{CH21}, Couzini\'e and Hirsch consider the sub-linear WARM model. 
They show that, for bounded degree (possibly infinite graphs), 
if the reinforcement is sufficiently weak, or if $\mathcal G = \mathbb Z$,
then the limiting edge-weights exist and are deterministic.}
In~\cite{BBCL}, and later~\cite{CL} and~\cite{Lima}, 
the cases of sub-linear, super-linear and linear reinforcement are considered.
In the linear case, if the original graph $\mathcal G$ is regular and not bipartite, 
then the vector of normalised weights converges almost surely 
to a {\it non-deterministic} limit.

Given these examples from the literature, it is quite surprising that in the ant process, with linear 
reinforcement, the vector of normalised weights always converges to a {\it deterministic} limit.
In fact, this difference of behaviour can be observed when looking at the simplest probabilistic model with linear reinforcement: (generalised) P\'olya urns.
A $d$-colour P\'olya urn of replacement matrix $R = (R_{i,j})_{1\leq i,j\leq d}$ is defined as follows:
at time zero, there is one ball of each colour in the urn, and
at every time step, we pick a ball uniformly at random among the balls in the urn, and if its colour was $i$, we return it to the urn together with $R_{i,j}$ balls of colour $j$ ($\forall 1\leq j\leq d$).
If $R$ is the identity, it is known (see Markov~\cite{Markov}) that the vector $\hat U(n)$ whose coordinates are the number of balls of each colour divided by~$n$ converges almost surely to a Dirichlet(1, \ldots, 1)-distributed random variable. 
In contrast, if the matrix $R$ is irreducible, then $\hat U(n)$ converges almost surely to a deterministic limit (see Janson~\cite{Janson04} and the references therein).

{Interestingly, in {a} not-yet-available paper~\cite{HK+} (whose results have been announced), 
Holmes and Klepstyn also exhibit a linear reinforcement model which, 
on some graphs, converges to a deterministic limit, a result that resonates with Conjecture~\ref{conj}.}

A similarity between this paper, the P\'olya urns with graph-based interaction of~\cite{BBCL} 
and the WARMs of~\cite{WARM} is the method of proof since 
we also use stochastic approximation. 
In particular, we use the ODE method for stochastic approximation, and apply results from Bena\"im~\cite{Benaim} and Pemantle~\cite{Pemantle}. 
However, the analysis of these stochastic approximations in the different examples of graphs we consider is quite different from the one in ~\cite{BBCL} and~\cite{WARM}. 
In fact, as we later explain in more details, 
each of our examples requires an ad-hoc argument, 
which suggest that proving our conjecture on the convergence of edge-weights in great generality is a difficult open problem.

That said, we prove a general result ensuring that on {\it any finite graph}, the sequence of normalised edge weights is a stochastic approximation, in a sense which is recalled in Definition~\ref{def.stoc.approx} below, with a vector field $F$ which is Lipschitz on a suitable convex compact subspace of the Euclidean space (see Proposition~\ref{prop.stoc.approx}).

\subsection{Definition of the model and main results}
Let $\mathcal G = (V,E)$ be a finite {(undirected)} graph with vertex set $V$ and edge set $E$, 
with two distinct marked nodes called $N$ (for ``nest'') and $F$ (for ``food'').
We define a sequence $({\bf W}(n) = (W_e(n))_{e\in E})_{n\geq 0}$ of random weights for the edges of $\mathcal G$ recursively as follows:
\begin{itemize}
\item At time zero, all weights are equal to~1, i.e.\ $W_e(0) = 1$ for all $e\in E$.
\item Given ${\bf W}(n)$, we sample a random walk $(X^{\sss (n+1)}_i)_{i\geq 0}$ 
on $\mathcal G$ according to the following distribution:
\begin{itemize}
\item the walk starts at node $N$, i.e.\ $X_0^{\sss (n+1)} = N$, 
\item it stops when first hitting $F$, i.e.\ $\mathbb P(X^{\sss (n+1)}_{i+1} = F | X^{\sss (n+1)}_{i}=F) = 1$ for all $i\geq 0$,
\item for all $i\geq 0$, for all {$u\in V\setminus\{F\}$}, $v\in V$, 
\[\mathbb P(X^{\sss (n+1)}_{i+1} = v | X^{\sss (n+1)}_{i}=u) = \frac{W_{\{u,v\}}(n)}{\sum_{u'\sim u}W_{\{u,u'\}}(n)} \bs 1_{\{u,v\}\in E},\]
where $u'\sim u$ if there is an edge linking the vertices $u'$ and $u$.
\end{itemize}
We let $\gamma(n+1)$ be the set of edges that were crossed at least once 
by the random walk $X^{\sss (n+1)}$; we call this the ``trace'' of the $(n+1)$-th walker.
For all $e\in E$, we set $W_e(n+1) = W_e(n)+\bs 1_{e\in \gamma(n+1)}$.
\end{itemize} 
We call this process the ``trace-reinforced'' ant process on $\mathcal G$.

In our first result, we focus on graphs that are ``tree-like'' in the following sense:

\begin{definition}
We say that a graph $\mathcal G = (V,E)$ with two marked nodes $N$ and $F$ 
is tree-like if the graph whose vertex set is $V\setminus \{F\}$ 
and whose edge set is $E$ minus all edges that contain~$F$ is a tree (i.e.\ a graph with no cycle).
\end{definition}

\begin{theorem}\label{th:main}
Assume that $\mathcal G = (V,E)$ is tree-like and that the edge $a=\{N,F\}$ belongs to $E$ with multiplicity~1. 
Then, almost surely when $n\to+\infty$,
\[\frac{W_a(n)}{n} \to 1
\quad\text{ and }\quad
\frac{W_e(n)}{n} \to 0, \text{ for all } e\in E\setminus\{a\}.\]
\end{theorem}
In other words, following this reinforcement algorithm, the ants eventually find the shortest path between their nest and the source of food, i.e.\ the proportion of ants that go from $N$ to $F$ by only crossing the edge $\{N,F\}$ is asymptotically equal to one, and the proportion of ants that cross any other edge asymptotically equals zero. 

Note that if the edge $\{N,F\}$ appears with multiplicity $\ell$ in $E$ (i.e.\ there are $\ell$ edges from $N$ to $F$ in parallel), then it is easy to deduce from Theorem~\ref{th:main} that the normalised weights of all other edges go to zero almost surely, and the weights of the $\ell$ edges from $N$ to $F$ converge almost surely, as a $\ell$-tuple, to a Dirichlet random variable with parameters $(1, \ldots, 1)$.

A natural extension of the set of tree-like graphs is the set of series-parallel graphs, which were considered for instance in \cite{HJ04} and in our previous paper \cite{KMS}. 
One could then ask whether the previous theorem extends to this class of graphs, that is, 
if the distance from the source to the food is one, do the weights of all edges not directly connected to both~$N$ and~$F$ go to zero?
Maybe surprisingly, the answer is no. Indeed our next result provides a counter-example, which is depicted by Figure~\ref{fig:cornet} and which we call the {\it cone} graph. 

\begin{figure}
\begin{minipage}[b]{.25\textwidth}\centering
\includegraphics[width=1.8cm]{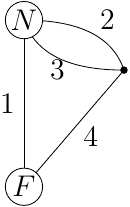}
\caption{The cone}
\label{fig:cornet}
\end{minipage}\hfill
\begin{minipage}[b]{.45\textwidth}\centering
\includegraphics[width=2.5cm]{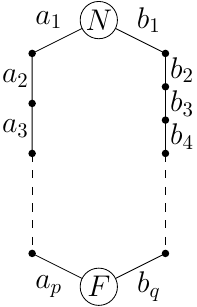}
\caption{The $(p,q)$-path graph of Theorem~\ref{th:two_paths}}
\label{fig:paths}
\end{minipage}\hfill
\begin{minipage}[b]{.25\textwidth}\centering
\includegraphics[width=2cm]{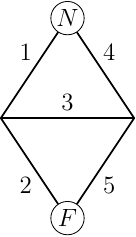}
\caption{The lozenge}
\label{fig:lozenge}
\end{minipage}
\end{figure}

\begin{proposition}\label{prop:cornet}
Let $\mathcal G$ be the graph of Figure~\ref{fig:cornet}. 
If we let $W_i(n)$ be the weight of edge~$i$ at time~$n$ (using the numbering of edges of Figure~\ref{fig:cornet}) and ${\bf W}(n) = (W_i(n))_{1\leq i\leq 4}$, then
almost surely when $n\to+\infty$,
\[\frac{{\bf W}(n)}n \to (1, \nicefrac13, \nicefrac13, 0).\]
\end{proposition}

The following result shows that Theorem~\ref{th:main} does not extend to tree-like graphs where $N$ and $F$ are at (graph-)distance at least~2 from each other. For two integers $p\geq 1$ and $q\geq 1$, we define the $(p,q)$-path graph as the graph 
with two parallel paths between $N$ and $F$, 
one of length $p$ and one of length $q$ (see Figure~\ref{fig:paths}).

\begin{theorem}\label{th:two_paths}
Let $\mathcal G = (V,E)$ be the $(p,q)$-path graph.
We let $a_1, \ldots, a_p$ (resp. $b_1, \ldots, b_q$) denote the edges of the path of length~$p$ (resp.\ $q$), numbered from the closest to the nest to the closest to the food.

If $\min(p,q)\geq 2$, then, almost surely for all $1\leq k\leq p$ and $1\leq\ell\leq q$,
\[\lim_{n\to+\infty} \frac{W_{a_k}(n)}n \to \alpha^k
\quad \text{ and }\quad
\lim_{n\to+\infty} \frac{W_{b_\ell}(n)}n \to \beta^\ell,\]
where $(\alpha, \beta)$ is the unique solution in $(0,1)^2$ of
\begin{equation}\label{eq:system}
\begin{cases}
\alpha^p+\beta^q = 1&\\
\alpha^p(1-\alpha) =\beta^q(1-\beta). 
\end{cases}
\end{equation}
\end{theorem}

Note that, if $p=q\geq 2$, the solution of~\eqref{eq:system} 
is explicit and given by $\alpha = \beta = 2^{-\nicefrac1p}$.

Now to give further support to our conjecture that normalised edge weights always converge to a deterministic limit (when the edges connected to $F$ are simple), we look at the lozenge graph in Figure~\ref{fig:lozenge}; this example, which was also considered in~\cite{KMS}, 
is different from all other cases so far, in the sense that it does not belong to the class of ``series-parallel'' graphs. 
\begin{proposition}\label{prop:lozenge}
Let $\mathcal G$ be the lozenge graph of Figure~\ref{fig:lozenge}. If we let $W_i(n)$ denote the weight of edge $i$ at time~$n$ (with the edges numbered as in Figure~\ref{fig:lozenge}), and ${\bf W}(n) = (W_i(n))_{1\leq i\leq 5}$, then almost surely as $n\to+\infty$, we have
\[\frac{{\bf W}(n)}{n} \to (w^*, \nicefrac12, \nicefrac12, w^*, \nicefrac12),\]
where $w^*$ is the unique solution of $2x^3+4x^2-2x-\nicefrac32 = 0$ in $(0,1)$.
\end{proposition}
{The fact that the limiting weights of edges $2$ and $5$ are equal to $\nicefrac12$ should not come as a surprise. Indeed, by symmetry they must be equal (assuming they are deterministic), 
and since each ant reinforces exactly one of these two edges at each step, their common value has to be $\nicefrac12$. 
Similarly for the edges~$1$ and~$4$, except that since each ant can reinforce both of them, 
one has now $w^*>1/2$. 
However, the fact that the limiting weight of edge~$3$ equals $\nicefrac12$ seems to be merely a coincidence.}
 
\medskip 
\noindent {\bf Notation:} Given some filtration $(\mathcal F_n)_{n\ge 0}$, and $Z$ some random variable, we will use the notation $\mathbb E_n[Z]$ to denote the conditional expectation of $Z$ with respect to $\mathcal F_n$.

\medskip
{\bf Acknowledgements:} 
We thank three anonymous referees for their careful reading of the paper and their constructive comments.  
Preliminary investigations on the problems treated in this paper were carried out by Yassine Hamdi,
a student at \'Ecole Polytechnique at the time, 
during an undergraduate research internship at the University of Bath, 
under the supervision of CM (see~\cite{Yassine} for Yassine's internship report).
The authors are grateful to Yassine for the time he spent on these questions, 
and to both the \'Ecole Polytechnique and the University of Bath for making this internship possible.

\section{Preliminaries}\label{sec.prelim}

\subsection{Urn processes}
We state here a result concerning (generalized) P\'olya urn processes.  
Given a function $G:[0,1]\to [0,1]$, we call {\it $G$-urn process}, a process $(X_n)_{n\ge 0}$ with integer values, such that almost surely $X_{n+1} 
\in \{X_n,X_n+1\}$, and for all $n\ge 0$, 
\[\mathbb P(X_{n+1} = X_n +1 \mid X_0,\dots,X_n) = G(\hat X_n),\] 
with $\hat X_n :=\frac{X_n}{n+2}$.
In general we will assume that it starts from $1$ at time $0$, i.e. that $X_0=1$, but we shall also consider other initial condition{s}. 
We then say that it starts from some value $k$ at time $m$, if we condition the process on the event $\{X_m=k\}$.

Informally $X_n$ corresponds 
to the number of (say) red balls after~$n$ draws in a P\'olya urn with two colours,
where at each step, we draw a ball in the urn at random, and replace it into the urn with an additional ball of the same colour. At each draw, the probability to pick a red ball is $G(p)$
if the proportion of red balls in the urn is $p$.

We will need the following standard result (which follows for instance from Corollary 2.7 and Theorem~2.9 in \cite{Pemantle}). 
\begin{proposition}\label{lem.urne}
Let $(X_n)_{n\ge 0}$ be a $G$-urn process, with $G$ a $C^1$-function. Then almost surely $(\hat X_n)_{n\ge 0}$ converges towards a stable fixed point  
of $G$, that is a (possibly random) point $p\in [0,1]$, such that $G(p)= p$ and $G'(p)\le 1$.

In particular if there exists $c>0$, such that $G(x)>x$, for all $x\in (0,c)$ (resp. $G(x)<x$ for all $x\in (1-c,1)$), 
then almost surely $\liminf_{n\to \infty} {\hat X_n}\ge c$ (resp. $\limsup {\hat X_n} \le 1-c$). 
\end{proposition}
We shall also use the following corollary. 
\begin{corollary}\label{prop.urne}
Let $(X_n)_{n\ge 0}$ be an integer valued process adapted to some filtration $(\mathcal F_n)_{n\ge 0}$, such that almost surely for all $n\ge 0$, $X_{n+1} \in \{X_n,X_n+1\}$, $X_0=1$, and for some function $G:[0,1]\to [0,1]$, 
\begin{equation}\label{Gurn.dom}
\mathbb P(X_{n+1} = X_n +1 \mid \mathcal F_n) \ge G(\hat X_n),
\end{equation}
with $\hat X_n :=\frac{X_n}{n+2}$. 
If there exists $\eta, c>0$, such that $G(x) > (1+\eta)x$, for all $x\in (0,c)$, 
then almost surely $\liminf_{n\to \infty} {\hat X_n} \ge c$. 
\end{corollary}
\begin{proof} For $\varepsilon \in (0,\eta)$, consider $G_\varepsilon:[0,1] \to [0,1]$, a $C^1$ function such that 
$x < G_\varepsilon(x) \le (1+\eta)x$, for all $x\in (0,c-\varepsilon)$, $G_\varepsilon(x) \le (1+\eta)x$, for $x\in (c-\varepsilon,c)$, and $G_\varepsilon \equiv  0$ on $[c,1]$. By assumption on $G$, 
one has $G(x) \ge G_\varepsilon(x)$, for all $x\in [0,1]$. It follows that  $(X_n)_{n\ge 0}$ stochastically dominates a $G_\varepsilon$-urn process, and 
applying {Proposition}~\ref{lem.urne}, we deduce that almost surely $\liminf \hat X_n \ge c-\varepsilon$. 
Since this holds for all $\varepsilon \in (0,c)$, the result follows. 
\end{proof}
In our applications, the process $(X_n)_{n\ge 0}$ will often be one coordinate of a higher-dimensional process $({\bf X}_n)_{n\ge 0}$, and 
$(\mathcal F_n)_{n\ge 0}$ will simply be the natural filtration of the process $({\bf X}_n)_{n\ge 0}$.

\subsection{Stochastic approximation and the ODE method}\label{sec.stoch.approx}
We use the following definition for a stochastic approximation (note that we do not seek for the most general definition here, but it will be sufficient for our purpose). 

\begin{definition}\label{def.stoc.approx} A {\it stochastic approximation} is a process $(X_n)_{n\ge 0}$, adapted to some filtration $(\mathcal F_n)_{n\ge 0}$, with values in a convex compact subset $\mathcal E\subseteq \mathbb R^d$, for some $d\ge 1$, that satisfies an equation of the type 
\begin{equation}\label{eq:def_algosto}
X_{n+1} = X_n + \frac{F(X_n) + \xi_{n+1}+ r_n}{n+2}, \qquad \text{for all }n\ge 0,
\end{equation}
where the vector field $F: \mathcal E\to \mathbb R$ is some Lipschitz function, the {\it noise} $\xi_{n+1}$ is $\mathcal F_{n+1}$-measurable and satisfies $\mathbb E_n[\xi_{n+1}] = 0$, for all $n\ge 0$, and the remainder term $r_n$ is $\mathcal F_n$-measurable and satisfies almost surely $\|r_n\| \le C/n$, for some deterministic constant $C>0$. 
\end{definition}
\begin{remark} 
The fact that we assume $\mathcal E$ to be a convex compact subset of $\mathbb R^d$ enables to easily extend $F$ into a Lipschitz continuous function defined on $\mathbb R^d$, simply by composing it with the orthogonal projection on~$\mathcal E$. 
We then fall into the setting of Bena\"im \cite{Benaim}, and we can rely on its results. Thus in the following, we will identify $F$ with its Lipschitz extension on $\mathbb R^d$, as defined here. 
\end{remark}

{\begin{remark} 
The choice of renormalisation factor equal to $\frac{1}{n+2}$ 
in Equation~\eqref{eq:def_algosto} is arbitrary and can be replaced by $\frac 1{n+3}$ (as in Sections~\ref{sec:cornet} and~\ref{sec:lozenge}). 
In Sections~\ref{sec:cornet} and~\ref{sec:lozenge}, we explain why the chosen renormalisation is the most convenient. 
\end{remark}}

The idea underlying the ODE method is that the trajectories of a stochastic approximation {\it asymptotically follow the solutions} of the differential equation
\begin{equation}\label{eq:ODE}
{\bs{\dot y}} =F({\bs y}).
\end{equation}
We recall that if for $x\in \mathbb R^d$, we let $(\Phi_t(x))_{t\ge 0}$ be the (unique because $F$ is Lipschitz) solution of \eqref{eq:ODE} starting at $x$, 
then this defines a {\it flow}, in the {following sense}. 

\begin{definition}
Let $\mathcal M$ be some metric space. 
A flow (or semi-flow) on $\mathcal M$ is an application $\Phi:\mathbb R_+\times \mathcal M \to \mathcal M$, such that $\Phi_0 = Id$, and $\Phi_{t+s}(x) = \Phi_t \circ \Phi_s(x)$, for all $s,t\ge 0$, and $x\in \mathcal M$. 
\begin{itemize}
\item A subset $\mathcal A\subset \mathcal M$ is said {to be} {\it invariant}, if $\Phi_t(x) \in \mathcal A$, for all $x\in \mathcal A$ and all $t\ge 0$. 
\item An {\it attractor} is a set $\mathcal A$ that admits a neighbourhood $\mathcal U\subset \mathcal M$, such that
\[\cap_{t\geq 0}\overline{\cup_{s>t} \Phi_s(\mathcal U)} = \mathcal A.\]
\end{itemize}
\end{definition}
We will frequently use the following result due to Bena\"im \cite[Prop.\ 4.1, Rk.\ 4.5, Prop.\ 5.3, Th.\ 5.7]{Benaim} (see also e.g., \cite[Prop.\ 2.10 and Th.\ 2.15]{Pemantle}). 
\begin{theorem}\label{th:pemantle}
Let $(X_n)_{n\geq 1}$ be a stochastic approximation. 
If there exists a deterministic constant $C>0$, such that almost surely $\sup_{n\geq 1}\|\xi_n\| \le C$,  
then almost surely, the limiting set $L(X) = \cap_{n\geq 0} \overline{\cup_{m\geq n} \{X_m\}}$ is invariant by the flow of the ODE~\eqref{eq:ODE}, connected, and the flow of the ODE restricted to $L(X)$ admits no other attractor than $L(X)$ itself.
\end{theorem}

We will also use the following corollary of Theorem~\ref{th:pemantle}:
\begin{corollary}\label{cor:pemantle}
Under the assumptions of Theorem~\ref{th:pemantle}, we let $\mathcal K\subset \mathcal E$ be invariant by the flow  of the ODE~\eqref{eq:ODE}.
We assume that there {exists} a set~$\mathcal U\subseteq \mathcal E$ such that
\begin{enumerate}[{\rm (i)}]
\item almost surely, $L(X)\subseteq \mathcal U$, 
\item for all $\bs w \in \mathcal U$, the solution of the ODE~\eqref{eq:ODE} started at $\bs w$ converges to $\mathcal K$, and
\item 
there exists an open set $\mathcal U'\subseteq \mathcal U$ such that $\mathcal K\subseteq \mathcal U'$ and the convergence to $\mathcal K$ is uniform on $\mathcal U'$, in the sense that 
$\cap_{t\geq 0}\overline{\cup_{s>t} \Phi_s(\mathcal U')} \subseteq \mathcal K$,
\end{enumerate}
then $L(X) \subseteq \mathcal K$ almost surely.
\end{corollary}
\begin{proof}
If $L(X) \not\subseteq \mathcal K$, then there exists a (possibly random) point $\bs x\in L(X)\setminus\mathcal K$. 
We show that, in this case, $L(X) \cap \mathcal K$ contains an attractor of the ODE restricted to $L(X)$, which concludes the proof by Theorem~\ref{th:pemantle}.
First, because $L(X)$ and $\mathcal K$ are both invariant under the flow of the ODE, so is $L(X)\cap \mathcal K$.
By (iii) and because $L(X)$ is invariant by the flow of the ODE, then
\[\cap_{t\geq 0}\overline{\cup_{s>t} \Phi_s(L(X) \cap\mathcal U')} \subseteq L(X) \cap\mathcal K.\]
Thus,
$\cap_{t\geq 0}\overline{\cup_{s>t} \Phi_s(L(X) \cap\mathcal U')}$ is an attractor of the ODE restricted to $L(X)$ and it is contained in $L(X) \cap\mathcal K$,
which concludes the proof.
%
\end{proof}

\subsection{The process of edge weights seen as a stochastic approximation}
Consider a finite graph $\mathcal G=(V,E)$, with two marked vertices $N$ and $F$. Recall that ${\bf W}(n)=(W_e(n))_{e\in E}$ denotes the {sequence of the} weights of the edges of the graph after $n$ steps of the trace-reinforced ant process, and let $\mathcal F_n: = \sigma({\bf W}(0),\dots,{\bf W}(n))$.

For any edge $e\in E$, and any $n\ge 0$, we let $X_e(n) := \frac{W_e(n)}{n+1}$, and $\bs X(n)= (X_e(n))_{e\in E}$. 
Next for any $\bs w \in [0,1]^E$, and any $e\in E$, we let $p_e({\bs w})$ 
be the probability that the edge $e$ belongs to the trace of a random walk on the graph $\mathcal G$ endowed with the weights $\bs w$, starting from $N$ and killed at $F$. Then we define 
$F:[0,1]^E\to [0,1]^E$, by  
\begin{equation}\label{def.F}
F_e(\bs w) := p_e(\bs w) - w_e,\quad \text{for any }e\in E.
\end{equation}
Given $\bs w \in [0,1]^E$, we set   
\[\pi_{\bs w}(x): = \sum_{e\sim x} w_e, \quad \text{for all } x\in V, 
\]
where $e\sim x$ means that we sum over all edges $e\in E$ that {have $x\in V$ as endpoint}, and recall that this defines a reversible measure for the random walk on $\mathcal G$ endowed with the weights $\bs w$.
We also let $\mathfrak S(\mathcal G)$ be the number of self-avoiding paths from $N$ to $F$ in $\mathcal G$, which we number in some arbitrary order: $\mathfrak c_1,\dots, \mathfrak c_{\mathfrak S(\mathcal G)}$. 
For $i=1,\dots, \mathfrak S(\mathcal G)$, we define 
\[\mathcal E_i:=  \Big\{\bs w\in [0,1]^E \, : \,  \pi_{\bs w}(N)\ge 1, \text{ and } w_e\ge \frac{1}{\frak S(\mathcal G)} \ \text{for all }e\in \frak c_i\Big\}. \] 
 Note that each $\mathcal E_i$ is a convex compact subset of $[0,1]^E$. Then we further define, 
\begin{equation}\label{def.E}
 \mathcal E:= \text{conv}\Bigg(\bigcup_{i=1}^{\frak S(\mathcal G)} \mathcal E_i\Bigg) 
 = \left\{\sum_{i=1}^{\frak S(\mathcal G)} \lambda_i \bs w_i \, : \, \begin{array}{ll} 
 \sum_i \lambda_i = 1, \text{ and } \lambda_i\ge 0, \text{ for all }i \\
\bs w_i\in \mathcal E_i, \text{ for all } i 
\end{array} \right\},
\end{equation}
 the convex hull of the union of the $\mathcal E_i$'s, which is also a convex compact subset of $[0,1]^E$.   
One has the following general fact. 

\begin{proposition}\label{prop.stoc.approx}
The function $F$ is Lipschitz on the space $\mathcal E$. Furthermore the process $(\bs X(n))_{n\ge 0}$ is a stochastic approximation on $\mathcal E$. More precisely, 
\begin{equation}\label{stoc.algo.weights}
\bs X(n+1) = \bs X(n) + \frac{1}{n+2}(F(\bs X(n)) + \bs \xi(n+1)), 
\end{equation}
where for any $e\in E$, $\xi_e(n+1) := {\bf 1}\{W_e(n+1) = W_e(n) +1\} - p_e(\bs X(n))$.  
\end{proposition}
\begin{remark} We stress that the proof of this result works in a wider setting, including the two variants of the process considered in our previous paper \cite{KMS}.  
\end{remark}
\begin{proof}
For the first part, we use a coupling argument. 
For all $\bs w, \bs w' \in \mathcal E$, 
we define $(X_n)_{n\ge 0}$ as the random walk on $\mathcal G$ 
equipped with edge-weights $\bs w$, 
and $(X'_n)_{n\ge 0}$ as the random walk on $\mathcal G$ equipped 
with edge-weights~$\bs w'$. 
Both walks start at $N$ and are killed when they first reach~$F$.
We couple $(X_n)_{n\ge 0}$ and $(X'_n)_{n\ge 0}$ until the first time when they differ, 
in a way that maximises the probability that they stay equal after each step.
We let $\tau$ be the random time when the walks first differ.
If $\tau_F$ denotes the first time when $(X_n)_{n\geq 0}$ hits~$F$, then
\ban
 \|F(\bs w) - F(\bs w')\|_\infty 
 &=  \max_{e\in E} |F_e(\bs w) - F_e(\bs w')| \notag\\
 &  \le \max_{e\in E} |p_e(\bs w) - p_e(\bs w')| +\|\bs w-\bs w'\|_\infty 
 \le  \mathbb P(\tau \le \tau_F)+\|\bs w-\bs w'\|_\infty \notag \\
 &\le  \sum_{k\ge 0} \mathbb P(X_k = X'_k, \, X_{k+1}  \neq  X'_{k+1},\, k<\tau_F ) {+\|\bs w-\bs w'\|_\infty}\notag\\
 &= \frac 12 \sum_{x\in V}  \sum_{k\ge 0} \mathbb P(X_k = X'_k=x,\, k<\tau_F) \sum_{e\sim x} \left|\frac{w_e}{\pi_{\bs w}(x)} - \frac{w'_e}{\pi_{\bs w'}(x)} \right|{+\|\bs w-\bs w'\|_\infty}, \label{eq:coupl}
\ean
where {the first inequality is obtained from \eqref{def.F} using the triangle inequality, the second inequality is a general fact which holds for any coupling of the two walks, since if an edge is crossed by only one of the two walks, then necessarily $\tau\le \tau_F$, and} the last equality in~\eqref{eq:coupl} holds because our coupling maximises the probability 
of the two walks staying equal. From~\eqref{eq:coupl}, we get
\begin{eqnarray*}
\|F(\bs w) - F(\bs w')\|_\infty  &\le & \frac 12  \sum_{x\in V}  \sum_{k\ge 0} \mathbb P(X_k = X'_k=x,\, k<\tau_F) \sum_{e\sim x}\left(\frac{|w_e-w'_e|}{\pi_{\bs w}(x)} + w'_e \frac{|\pi_{\bs w'}(x) - \pi_{\bs w}(x)|}{
\pi_{\bs w}(x)\cdot \pi_{\bs w'}(x) }\right)\\
&&{+\|\bs w-\bs w'\|_\infty} \\
& \le & \frac 12  \sum_{x\in V}  \sum_{k\ge 0} \mathbb P(X_k =x,\, k<\tau_F) \left(\frac{\sum_{e\sim x} |w_e-w'_e|}{\pi_{\bs w}(x)} +  \frac{|\pi_{\bs w'}(x) - \pi_{\bs w}(x)|}{
\pi_{\bs w}(x)}\right) \\
&&{+\|\bs w-\bs w'\|_\infty}\\
& \le &  \sum_{x\in V}  \frac{G_{\bs w}(N,x)}{\pi_{\bs w}(x)} \cdot \sum_{e\sim x} |w_e-w'_e|{+\|\bs w-\bs w'\|_\infty},  
 \end{eqnarray*} 
with $G_{\bs w}(\cdot, \cdot)$ the Green's function on the graph $\mathcal G$ endowed with the weights $\bs w$ (i.e.\ the mean number of visits to the second argument for a random walk starting 
from the first argument, up to its hitting time of $F$).
Using the reversibility of the measure $\pi_{\bs w}$, we deduce that (see e.g. \cite[Exercise 2.1(e)]{LP}),
\[\frac{G_{\bs w}(N,x)}{\pi_{\bs w}(x)} = \frac{G_{\bs w}(x,N)}{\pi_{\bs w}(N)} \quad (\forall x\in V).\]
Using also that $ G_{\bs w}(x,N) \le G_{\bs w}(N,N)$, we get 
\[
 \|F(\bs w) - F(\bs w')\|_\infty 
 \le  \frac{G_{\bs w}(N,N)}{\pi_{\bs w}(N)}\sum_{x\in V} \sum_{e\sim x} |w_e-w'_e| {+\|\bs w-\bs w'\|_\infty}
 \le   {\left(1+\frac{2G_{\bs w}(N,N)}{\pi_{\bs w}(N)}\right)}\cdot \|\bs w - \bs w'\|_1,
 \]
 with $\|\bs w-\bs w'\|_1 =  \sum_{e\in E} |w_e - w'_e|$. 
Now by definition, for $\bs w\in \mathcal E$, one has $\pi_{\bs w}(N) \ge 1$, and we claim that $G_{\bs w}(N,N)$ is also bounded by a positive constant independent of $\bs w$ (only depending on the graph~$\mathcal G$). Indeed, by \cite[Eq.\ (2.4)]{LP}, 
for a random walk starting from $N$, 
the number of returns to $N$ before hitting $F$ 
is a geometric random variable with mean $\pi_{\bs w}(N) / \mathcal C_{(\mathcal G,\bs w)}(N,F)$, 
where $\mathcal C_{(\mathcal G,\bs w)}(N,F)$ denotes the effective conductance 
between $N$ and $F$ in the graph $\mathcal G$ endowed with the weights $\bs w$.
Moreover, by definition of $\mathcal E$, for any $\bs w \in \mathcal E$, there exists a self-avoiding path from $N$ to $F$ such that all the edges on this path have a weight larger than 
$\frak S(\mathcal G)^{-2}$ (we recall that $\frak S(\mathcal G)$ denotes the number of self-avoiding paths between~$N$ and~$F$ in~$\mathcal G$). 
Such a path has an effective conductance larger than $(h_{\max}(\mathcal G)\cdot \frak S(\mathcal G)^2)^{-1}$, 
where $h_{\max}(\mathcal G)$ denotes the maximal length 
of a self-avoiding path from $N$ to $F$ in~$\mathcal G$. 
By Rayleigh's monotonicity principle, we also have $\mathcal C_{(\mathcal G,\bs w)}(N,F)
\geq (h_{\max}(\mathcal G)\cdot \frak S(\mathcal G)^2)^{-1}$.
Finally, note that $\pi_{\bs w}(N)$ is bounded by the degree of~$N$, say $d_{\mathcal G}(N)$.
In total, this implies that, for all $\bs w,\bs w'\in \mathcal E$,
\[\|F(\bs w) - F(\bs w')\|_\infty \le K(\mathcal G)\cdot \|\bs w- \bs w'\|_1,\]
with $K(\mathcal G):={1+} 2 \left(1+d_{\mathcal G}(N)\cdot h_{\max}(\mathcal G)\cdot \frak S(\mathcal G)^2\right)$, a constant which only depends on the graph $\mathcal G$. 
This concludes the proof of the fact that $F$ is Lipschitz on $\mathcal E$. 

Since~\eqref{stoc.algo.weights} is straightforward by definition of the model, 
it only remains to show that $\bs X(n)$ belongs to $\mathcal E$, for all $n\ge 0$. 
The fact that $\pi_{\bs X(n)}(N)\ge 1$ follows from the fact that, by definition of the model, 
at each step at least one of the edges incident to $N$ is reinforced.
Furthermore, at each step, at least one of the self-avoiding paths from $N$ to $F$ is reinforced, 
which implies that, at any time~$n\ge 0$, at least one of these self-avoiding paths has been reinforced at least ${n}/{\frak S(\mathcal G)}$ times.
In other words, for all $n\geq 0$, $\bs X(n)$ belongs to at least one of the $\mathcal E_i$'s, and thus $\bs X(n)$ belongs to $\mathcal E$ as claimed.
\end{proof}

\subsection{Case when $N$ and $F$ are at distance one}
In this section, we prove the following general fact: 
if $N$ and $F$ are at distance~$1$, 
then the only simple paths from $N$ to $F$ which ``survive'' asymptotically 
are those of length one.
In other words, asymptotically, almost all of the ants reach~$F$ by last crossing one of the edges that connect $N$ and $F$.
However, it is not true in general that the only edges which survive are those from $N$ to $F$; the cone graph of Proposition~\ref{prop:cornet} is a counter-example. 
\begin{proposition}\label{prop:liminf_NF=1}
Assume that $\mathcal G$ is a finite graph with two marked vertices $N$ and $F$, connected by at least one edge.  Let $({\bf W}(n))_{n\geq 0}$ be the process of edge-weights of the trace-reinforced ant-process on $\mathcal G$.
Then for any edge $e$ connected to $F$ but not to $N$, one has $W_e(n)/n \to 0$ almost surely. 
\end{proposition}

\begin{proof} 
First note that it is enough to prove the result in the case when there is a unique edge $a = \{N,F\}$ (i.e.\ it has multiplicity one). 
Indeed, if $\{N,F\}$ has multiplicity $m\geq 2$, then, by definition of the ant process, 
at most one of these $m$ edges belongs to the trace of each ant.
Hence, the process obtained by adding the weights of theses edges into one weight is the ants process on the graph in which the $m$ parallel edges have been merged into one edge {with initial weight $m$}.

Assume that $\bs w$ is such that $w_a \notin \{0,1\}$.
Let $\mathcal G'$ be the graph obtained by removing edge $a$ from $\mathcal G$:
i.e.\ $\mathcal G' = (V, E')$ where $E' = E\setminus\{a\}$. We equip the edges of $\mathcal G'$ with the weights $(w_e)_{e\in E'}$, and let $\mathcal C_{\mathcal G'}(\bs w)$ denote the conductance between $N$ and $F$ in $\mathcal G'$ equipped with these weights. We denote $p_a({\bs w})$ the probability to reinforce the edge $a$ when the weights over the graph are given by ${\bs w}$.
With this notation, we have
\[p_a(\bs w) 
= \frac{w_a}{w_a + \mathcal C_{\mathcal G'}(\bs w)}.
\]
We let $k$ denote the number of edges connected to $N$ in $\mathcal G'$, 
and $\ell$ denote the number of edges connected to $F$ in $\mathcal G'$.
We define the graph $\mathcal{G}(k,\ell)$ as the graph with vertex set $\{N, F, P\}$
and with edge set $\{N,P\}$ with multiplicity $k$ and $\{P,F\}$ with multiplicity $\ell$ 
(see Figure~\ref{fig:Gkl}).
We equip the $k$ edges between $N$ and $P$ in $\mathcal{G}(k,\ell)$ 
with the same weights as the $k$ edges connected to $N$ in $\mathcal G'$,
and the $\ell$ edges between $P$ and $F$ in $\mathcal{G}(k,\ell)$ 
with the same weights as the $\ell$ edges connected to $F$ in $\mathcal G'$.
The graph $\mathcal{G}(k,\ell)$ can be obtained from $\mathcal{G}'$ by merging all vertices different from $N$ and $F$ into one node called $P$, 
or equivalently by {adding edges between all pair of vertices disjoint of $N$ and $F$, and} assigning an infinite weight to all edges not connected to $N$ or $F$.
By Rayleigh's monotonicity principle (see \cite[{\S2.4}]{LP}), 
the conductance $\mathcal{C}_{\mathcal{G}(k,\ell)}({\bs w})$ of $\mathcal{G}(k,\ell)$ 
is at least equal to the conductance of~$\mathcal{G}'$.

\begin{figure}
\begin{center}\includegraphics[width=3cm]{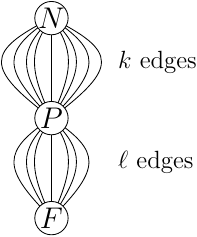}\end{center}
\caption{The graph $\mathcal G(k,\ell)$.}
\label{fig:Gkl}
\end{figure}

The conductance of $\mathcal G(k,\ell)$ with these weights is given by
\[\mathcal C_{\mathcal G(k,\ell)}(\bs w) = \frac{\displaystyle\sum_{e\in E'\colon N\sim e} w_e\sum_{e\in E'\colon F\sim e} w_e}{\displaystyle\sum_{e\in E'\colon N\sim e}w_e + \sum_{e\in E'\colon F\sim e} w_e}
\leq \frac{k (1-w_a)}{k+1-w_a},\]
where $X\sim e$ denotes that the vertex $X$ is an endpoint of $e$, where we have used that for all $e\in E'$, $w_e\leq 1$, and also that $\sum_{e\in E\colon F\sim e}w_e = 1$, and thus \[\sum_{e\in E'\colon F\sim e} w_e = 1- w_a.\]
Therefore, 
\[p_a(\bs w)\geq\frac{w_a}{w_a +\frac{k (1-w_a)}{k+1-w_a}}.\]
Thus, $(W_a(n))_{n\ge 0}$ stochastically dominates a $G$-urn process with
\[G(x)=\frac x{x +\frac{k (1-x)}{k+1-x}}.\]
Note that $G(x)=x$ if and only if $x\in\{0,1\}$, and one can compute that $G'(0){=(k+1)/k}>1$. Thus {Corollary}~\ref{prop.urne} shows that $W_a(n)/n$ converges almost surely to $1$, and as a consequence one also has that $W_e(n)/n$ converges almost surely to 0 for all $e$ connected to $F$ different from $a$.
This concludes the proof.
\end{proof}

\section{Proof of Theorem~\ref{th:main}}
First note that Proposition~\ref{prop:liminf_NF=1} implies that $W_a(n)/n \to 1$, almost surely.  
Note also that by definition of the model, it now suffices to show that the normalised weight of all the other edges connected to $N$ go to zero almost surely, as the weight of any edge $e$ in the tree is always smaller than the weight of the unique edge connected to $N$ on the path from $e$ to $N$.

So let $e$ be some edge connected to $N$, which is different from $a$. 
By assumption, the graph $\mathcal G$ is a tree rooted at $N$ and whose leaves have been merged into $F$. Thus, since the ants are stopped when first hitting $F$, the first time each ant crosses the edge $e$ has to be from $N$ to the other extremity of $e$. Moreover, if an ant crosses the edge $a$ it is stopped immediately. 
This implies that the probability to cross $e$ before reaching $F$ on $\mathcal{G}$ 
is smaller than on the graph consisting only of the two edges $a$ and $e$ (in parallel between $N$ and $F$). This gives for all $n\ge 0$, almost surely, 
\begin{equation}\label{WeWa}
\mathbb P(W_e(n+1) = W_e(n) + 1 \mid {\bf W}(n)) \le \frac{W_e(n)}{W_e(n) + W_a(n)}=\frac{\hat W_e(n)}{\hat W_e(n) + \hat W_a(n)}, 
\end{equation}
where we recall that we write $\hat W(n) = W(n)/(n+2)$.  
Now let $\varepsilon >0$ be fixed. We know that almost surely, for $n$ large enough, $\hat W_a(n) \ge 1-\varepsilon$. 
On the other hand on the event when $\hat W_a(n) \ge 1 - \varepsilon$, for all $n$ larger than some integer $n_0$, we know by \eqref{WeWa} that 
$(W_e(n))_{n\ge n_0}$ is dominated by a $G$-urn process with $G(x) = x/(x+1-\varepsilon)$. Then Proposition~\ref{lem.urne} shows that almost surely, $\limsup_{n\to \infty} \hat W_e(n) \le \varepsilon$. 
Since this holds for all $\varepsilon>0$, this concludes the proof.


\section{The cone}\label{sec:cornet}
We first note that by Proposition~\ref{prop.stoc.approx} and the specific feature{s} of the  cone, the process $\hat {\bf W}(n) := {\bf W}(n)/(n+2)$ is a stochastic approximation on the space
\[\mathcal E' = \{\bs w = (w_1, \ldots, w_4)\in \mathcal E \colon w_1 + w_4 = 1, w_4\leq w_2+w_3\},\]
with $\mathcal E$ as defined in~\eqref{def.E}. 
More precisely, for all $n\geq 0$, we have
\[\hat {\bf W}(n+1) = \hat{\bf W}(n) + 
\frac1{n+3}\big(F(\hat{\bf W}(n))+ \xi_{n+1}\big),\]
with $\xi_{n+1}$ some martingale difference, and 
for all $1\leq i\leq 4$, $F_i(\bs w) = p_i(\bs w) - w_i$, where $p_i(\bs w)$ is the probability that edge $i$ belongs to the trace of a random walk on the graph endowed with weights $\bs w$, which starts 
from $N$ and is killed at $F$. 

{\begin{remark}
Note that, for the cone graph, it is convenient to define $\hat{\bf W}(n)$ as ${\bf W}(n)/(n+2)$ rather than as ${\bf W}(n)/(n+1)$. This is because, by definition of the process, for all $n\geq 0$, $W_1(n) + W_4(n) = n+2$; indeed, $W_1(0) + W_4(0) = 2$ and, at each step we increase exactly one of $W_1(n)$ and $W_4(n)$ by one.
\end{remark}}

To calculate $p_2$ (and thus $p_3$, by symmetry),
we decompose according to the first step of the ant: to reinforce edge~$2$, it has to either go straight through edge $2$ (in which case, no matter what it does later, edge $2$ will be reinforced), or go through edge $3$. In the latter case, the second step of the ant has to be either through edge $2$ (edge $2$ gets reinforced no matter what happens next) or back through edge $3$, in which case we start again. Hence,
\[p_2(\bs w) = \frac{w_2}{w_1+w_2+w_3} +  \frac{w_3}{w_1+w_2+w_3} \left(\frac{w_2}{w_2+w_3+w_4} +\frac{w_3}{w_2+w_3+w_4} \, p_2(\bs w)\right).\]
Solving this in terms of $p_2(\bs w)$ 
gives
\[p_2(\bs w) =  \frac{w_2(w_2+2w_3+w_4)}{(w_1+w_2+w_3)(w_2+w_3+w_4)-w_3^2}.\]
To calculate $p_1(\bs w)$, 
we use the electrical networks method (see, e.g.~\cite{LL}) to see that the probability, starting from~$N$, to cross edge~$1$ before edge~$4$ is~$w_1$ divided by the total conductance between~$N$ and~$F$. 
The latter conductance is given by $w_1 + \frac{(w_2+w_3)w_4}{w_2+w_3+w_4}$, which gives
\[p_1(\bs w)=\frac{w_1}{w_1 + \frac{(w_2+w_3)w_4}{w_2+w_3+w_4}}.\]
Thus, the coordinates of the function $F$ are given by
\begin{linenomath}\begin{align*}
F_1(\bs w) &= \frac{w_1}{w_1 + \frac{(w_2+w_3)w_4}{w_2+w_3+w_4}}-w_1
& F_2(\bs w) &= \frac{w_2(w_2+2w_3+w_4)}{(w_1+w_2+w_3)(w_2+w_3+w_4)-w_3^2}-w_2
\\
F_4(\bs w) &= -F_1(\bs w) 
&F_3(\bs w) &= \frac{w_3(2w_2+w_3+w_4)}{(w_1+w_2+w_3)(w_2+w_3+w_4)-w_2^2}-w_3.
\end{align*}\end{linenomath}
Now our aim is to use the ODE method and for this we proceed in 4 steps:

{\bf (1)} we first remind that $\lim {\hat W_1(n)} = 1$ almost surely as $n\to+\infty$, thanks to Proposition~\ref{prop:liminf_NF=1},

{\bf (2)} we prove that $\liminf_{n\to+\infty}{\hat W_2(n)\wedge \hat W_3(n)}>0$ (where $x\wedge y$ denotes the minimum of $x$ and $y$),

{\bf (3)} we prove that for any $w$ in $\mathcal U: = \{\bs w\in \mathcal E\colon w_1 = 1, w_2w_3\neq 0\}$, the solution of ${\bs{\dot y}} =F({\bs y})$ started at~$\bs w$ 
converges to $(1, \nicefrac13, \nicefrac13, 0)$, 

{\bf (4)} we finally apply Corollary~\ref{cor:pemantle}, together with (1), (2) and (3) to conclude.

\medskip
For {\bf (2)}, note that, 
if $\bs w\in \mathcal E'$, $w_1\to 1$, $w_2 \to 0$, and $w_4\to 0$,
then
\[({w_1}+w_2+w_3)(w_2+w_3+w_4)-w_3^2
\sim w_2+w_3+w_4 + w_2w_3 + w_3w_4
\sim w_2+w_3+w_4,\]
because $w_2w_3 = o(w_3)$ and $w_3w_4=o(w_3)$.
This implies
\ba
F_2(\bs w) 
&\sim \frac{w_2(w_2+2w_3+w_4)}{w_2+w_3+w_4}-w_2
= \frac{w_2(w_2+2w_3+w_4) - w_2(w_2+w_3+w_4)}
{w_2+w_3+w_4}
=\frac{w_2w_3}{w_2+w_3+w_4}.
\ea
Finally, since, for all $\bs w\in\mathcal E'$, $w_4\leq w_2+w_3$, we get that, as $w_1\to 1$, $w_4\to 0$ and $w_2\to 0$,
\[F_2(\bs w) \geq \frac{w_2w_3(1+o(1))}{2(w_2+w_3)} \geq \frac{w_2\wedge w_3}{{4}}(1+o(1)).
\]
In other words, there exists $\varepsilon>0$ such that, for all $\bs w\in\mathcal E'$ with $w_2\leq \varepsilon$, and $w_1\geq 1-\varepsilon$,
 \begin{equation}\label{eq:cornet_F2LB}
 F_2(\bs w)\geq  \frac{w_2\wedge w_3}{8}.
 \end{equation}
By symmetry, for all $\bs w\in\mathcal E'$ such that $w_3\leq \varepsilon$, and $w_1\geq 1-\varepsilon$,
\begin{equation}\label{eq:cornet_F3LB}
F_3(\bs w)\geq  \frac{w_2\wedge w_3}{8}.
\end{equation}
Thus, if we set $H(x) = {9x \bs 1_{x\leq\varepsilon}/8}$,
by Equations~\eqref{eq:cornet_F2LB}, and \eqref{eq:cornet_F3LB},
and since $\hat W_1(n)\to 1$ almost surely as $n\to\infty$, we get that,
 for all sufficiently large $n$,
the random variable $\Phi(n):=W_2(n)\wedge  W_3(n)$ satisfies almost surely on the event when $W_2(n)\neq W_3(n)$, 
\begin{equation}\label{Phi.acc}
\mathbb P\big(\Phi(n+1)= \Phi(n)+1\mid {\bf W}(n)\big) \geq H\big(\hat \Phi(n)\big), 
\end{equation}
with $\hat \Phi(n) := \Phi(n)/n$. 
This is however not sufficient to apply Corollary~\ref{prop.urne}, since when $W_2(n) = W_3(n)$,  one has $\Phi(n+1) = \Phi(n)+1$, only when both edges $2$ and $3$ 
are reinforced at the next step, which holds with smaller probability. 
To overcome this issue, we introduce the following quantity: 
\[\Psi(n) := \Phi(n) + \sum_{k=0}^{n-1} \bs 1\{W_2(k) = W_3(k), \, W_2(k+1) \neq W_3(k+1)\}. \] 
Note that contrarily to $\Phi$, the function $\Psi$ increases by one unit as soon as at least one of the two edges $2$ or $3$ is reinforced, whatever their weights are at this time (in particular this holds even when 
they are equal). 
Therefore, \eqref{eq:cornet_F2LB} and~\eqref{eq:cornet_F3LB}
imply that, almost surely for all sufficiently large $n$, 
\begin{equation}\label{Psi.acc}
\mathbb P\big(\Psi(n+1)= \Psi(n)+1\mid {\bf W}(n)\big) \geq H\big(\hat\Phi(n)\big). 
\end{equation}
We now claim that almost surely one has $\Phi(n) \sim \Psi(n)$ as $n\to+\infty$, 
which by Corollary~\ref{prop.urne}  
 implies $\liminf \Phi(n)>0$, because $H'(0)={\nicefrac 98}>0$.

In fact we prove a stronger statement: almost surely for all $n$ large enough, 
\begin{equation}\label{eq:PsiPhi}
\Psi(n) \le \Phi(n) +  \Phi(n)^{\nicefrac34}.
\end{equation}
To see this, set 
\[Z_n := \max(W_2(n) , W_3(n)) - \min(W_2(n),W_3(n))\quad (\forall n\geq 1).\] 
Note that, conditionally on the event that edge $2$ or $3$ is reinforced but not both, 
the one with largest weight is more likely to be reinforced than the other. 
This implies that the process $(Z_k)_{k\ge 0}$ taken at its jump times (i.e.\ the times $k=0$ and all $k\geq 1$ such that $Z_k\neq Z_{k-1}$)
stochastically dominates a simple random walk $(\mathrm{rSRW}_n)_{n\geq 0}$ 
on $\mathbb Z_+=\{0,1,\dots\}$ reflected at $0$.
We let $L(n)$ denote the number of times $(Z_k)_{k\ge 0}$ returns to zero before time~$n$, i.e.
\[L(n) = \sum_{k=0}^n \bs 1\{Z_k = 0\},\]
and $N(n)$ the number of jump times of $(Z_k)_{k\ge 0}$ during its first $L(n)$ excursions out of zero
(equivalently the number of times only one of the two edges $2$ or $3$ is reinforced before the last time before $n$ when the weights of edges $2$ and $3$ are equal).
Then, for all integers~$n$, $L(n)$ is stochastically dominated by the number of returns to the origin of $(\mathrm{rSRW}_k)_{k\geq 0}$ before time $N(n)$. 
Moreover, by definition,
\begin{equation}\label{L(n).1}
\Psi(n) - \Phi(n) = \sum_{k=0}^{n-1} \bs 1\{W_2(k) = W_3(k), \, W_2(k+1)\neq W_3(k+1)\} \le 1+ L(n).
\end{equation}
During each excursion out of the origin, 
the probability to hit level $N^{\nicefrac58}$ is equal to $N^{-\nicefrac58}$, by a standard Gambler's ruin estimate. Moreover, by Hoeffding's inequality, for any $N\ge 1$, 
the probability for a simple random walk started at level $N^{\nicefrac58}$ to hit $0$ before time $N$ is bounded by $2N\exp(-N^{\nicefrac14}/2)$. 
Therefore the probability that 
$(\mathrm{rSRW}_k)_{k\geq 0}$ returns more than $N^{\nicefrac34}/2$ times to the origin before time $N$ is bounded by $(1-N^{-\nicefrac58})^{N^{3/4}/2}+2N\exp(-N^{\nicefrac14}/2)  \le  \exp(-N^{\nicefrac18}/4)$, for all large $N$. 
Using {the} Borel-Cantelli lemma, we deduce that, almost surely for all large $N$, $(\mathrm{rSRW}_k)_{k\geq 0}$ returns at most $N^{\nicefrac34}/2$ times to the origin before time $N$.
This implies that, 
almost surely for all large $n$, 
\begin{equation}\label{L(n).2}
L(n) \le \frac12 N(n)^{\nicefrac 34}.
\end{equation} 
On the other hand, during each excursion of the process $(Z_k)_{k\ge 0}$ out of zero, the process $(\Phi(k))_{k\ge 0}$ increases by at least half the number of jumps made by $(Z_k)_{k\ge 0}$ during this excursion. It follows that $N(n)\le 2\Phi(n)$, for all $n\ge 0$, and together with \eqref{L(n).1} and \eqref{L(n).2}, this concludes the proof of~\eqref{eq:PsiPhi}, and thus of point ${\bf (2)}$.

\medskip
{\bf (3)}  Let $\bs w \in \mathcal U$, 
and let $\bs \Phi(t) = (\Phi_1(t), \Phi_2(t), \Phi_3(t), \Phi_4(t))$ be the solution of $\dot {\bs y} = F(\bs y)$ started at $\bs w$. We need to prove that $\bs \Phi(t)\to (1, \nicefrac13, \nicefrac13, 0)$.

\begin{figure}
\begin{center}
\includegraphics[width=8cm]{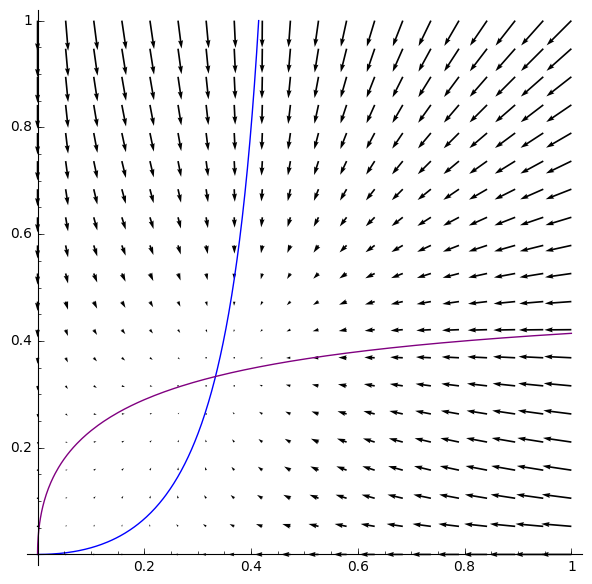}
\end{center}
\caption{The vector field $F(\bs w)$ on $\mathcal E\cap\{w_1 = 1\}$ in the cone case.}
\label{fig:champs}
\end{figure}

For all $\bs w\in \mathcal U$, 
we have
\[F_2(\bs w) = \frac{w_2(w_2+2w_3)}{(1+w_2+w_3)(w_2+w_3)-w_3^2}-w_2
\quad \text{ and }\quad
F_3(\bs w) = \frac{w_3(2w_2+w_3)}{(1+w_2+w_3)(w_2+w_3)-w_2^2}-w_3.
\]
Note that $F_2(\bs w) = 0$ if and only if $w_2 = 0$ or
\[w_2+2w_3 = (1+w_2+w_3)(w_2+w_3)-w_3^2
\quad\Leftrightarrow\quad
w_3 = w_2^2+2w_2w_3.\]
Similarly, $F_3(\bs w) = 0$ if and only if $w_3 = 0$ or
$w_2 = w_3^2+2w_2w_3$.
Thus, for all $\bs w\in\mathcal E\cap \{w_1 = 1\}$, $F(\bs w) = 0$ 
if and only if $w_2 = w_3 = 0$ or {$w_2w_3\neq0$ and}
\[\begin{cases}
w_3 = w_2^2+2w_2w_3\\
w_2 = w_3^2+2w_2w_3
\end{cases}
\quad
\Leftrightarrow\quad
\begin{cases}
w_3-w_2 = w_2^2-w_3^2\\
w_2 = w_3^2+2w_2w_3
\end{cases}
\Leftrightarrow\quad
\begin{cases}
w_3=w_2\\
1 = 3w_2
\end{cases}
\]
i.e.\ $w_2 = w_3 = \nicefrac13$.
Thus the only zeros of $F$ on $\mathcal E'\cap \{w_1 = 1\}$ are $(1, 0, 0, 0)$ and $(1,\nicefrac13, \nicefrac13, 0)$.
Similar calculations show that $F_2(\bs w)>0$ if and only if $w_2<{w_3^2}/(1-2w_3)$, and $F_3(\bs w)>0$ if and only if $w_3<{w_2^2}/(1-2w_2)$. In Figure~\ref{fig:champs}, we plot the vector field $(F_2(1, w_2, w_3, 0), F_3(1, w_2, w_3, 0))$ with $w_2$ on the horizontal axis and $w_3$ on the vertical axis. The blue curve is where $F_2=0$, the purple curve where $F_3 = 0$.
Note that $[0,1]^2\setminus(\{w_3={w_2^2}/(1-2w_2)\}\cup \{w_2={w_3^2}/(1-2w_3)\})$ has four connected components: on the bottom-left one, $F_2, F_3>0$, on the top-right one, $F_2, F_3<0$, on the bottom right one, $F_2<0$ while $F_3>0$, and, finally, on the top-left one, $F_2>0$ while $F_3<0$.

Thus, any solution of the ODE started in 
$\mathcal U$ converges to $(1,\nicefrac13, \nicefrac13, 0)$,
 as claimed.

\medskip
{\bf (4)} Finally, we prove that $\hat{\bf W}(n) \to (1,\nicefrac13, \nicefrac13, {0})$ almost surely as $n\to+\infty$. For this we apply Corollary~\ref{cor:pemantle}: in Steps (1) and (2), we have shown that
$\lim_{n\to+\infty} \hat W_1(n) = 1$ and $\liminf_{n\to+\infty} \hat W_2(n)\wedge \hat W_3(n)>0$.
This implies that almost surely the limiting set $L(\hat{\bf W})$ of the process $(\hat {{\bf W}}(n))_{n\ge 0}$ is contained in $\mathcal U$ and thus assumption (i) of Corollary~\ref{cor:pemantle} holds. 
Then {\bf (3)} implies that  assumption (ii) of Corollary~\ref{cor:pemantle} holds.
For assumption (iii), note that, on $\mathcal U$, the curves $F_2(\bs w) = 0$ and $F_3(\bs w) = 0$ intersect at $\bs p=(1,\nicefrac13, \nicefrac13, {0})$, and thus define four ``branches'' around $\bs p$. 
For all $t$ small enough ($t<\nicefrac1{10}$ is enough),
let $\bs x_t$ the point at distance $t$ on the ``left''-branch of the curve $F_2(\bs w) = 0$;
then let $\mathcal C_t$ be the rectangle 
whose four vertices are on each of the four branches (one of them being $\bs x_t$)
and whose four sides are parallel to either $(1, 1)$ or $(1, -1)$;
this rectangle exists because the curve $F_2(\bs w) = 0$ 
is the image of the curve $F_3(\bs w) = 0$ by the symmetry of axis $w_2 = w_3$.
In (3), we have shown that the vector field $F$ at any point of $\mathcal C_t$ points towards the interior of $\mathcal C_t$.
We now let $d(\bs w) = \inf \{t>0 \colon \bs w\in  \mathcal C_t\}$.
Note that $t\mapsto d(\Phi_t(\bs w))$ is non-increasing because the vector field $F$ at any point of $\mathcal C_t$ points towards the interior of $\mathcal C_t$.
This implies uniform convergence on $\mathcal U' = \cup_{t<\nicefrac1{10}} \mathcal C_t$.
Indeed, if we assume that the convergence is not uniform, then, there exists $\delta>0$ and $(s_n)_{n\geq 0}$ and $(\bs w_n)_{n\geq 0}$ such that, for all $n\geq 0$, $s_n\geq n$, $\bs w_n\in \mathcal U'$ and $d(\Phi_{s_n}(\bs w_n))\geq \delta$. 
By compactness, we can assume that $\bs w_n \to \bs w\in \overline{\mathcal U'}$ as $n$ goes to infinity.
For all $t\geq 0$, for all $n\geq t$, $d(\Phi_t(\bs w_n))\geq d(\Phi_{s_n}(\bs w_n))\geq \delta$.
By continuity, $d(\Phi_t(\bs w))\geq \delta$, which is impossible.

Thus, by Corollary~\ref{cor:pemantle}, this implies that $L(\hat{\bf W})= \{(1, \nicefrac13, \nicefrac13, 0)\}$, as desired.


\section{Two parallel paths: proof of Theorem~\ref{th:two_paths}}\label{sec:2paths}
Recall that, in Theorem~\ref{th:two_paths}, 
the graph $\mathcal G=(V,E)$ is the $(p+q)$-path graph of Figure~\ref{fig:paths}.
In this section, we assume that $\min(p,q)\geq 2$.
We let $a_1, \ldots, a_p$ denote the $p$ edges on one of the paths from $N$ to $F$ (ordered from $N$ to $F$, i.e.\ $a_1$ links to $N$ while $a_p$ links to $F$), and $b_1, \ldots, b_q$ denote the edges on the other path from $N$ to $F$ (also ordered from $N$ to $F$).

Finally, for all $e\in E=\{a_1, \ldots, a_p, b_1, \ldots, b_q\}$, we let $W_e(n)$ denote the weight of edge~$e$ at time $n$ and set ${\bf W}(n) = (W_{a_1}(n), \ldots, W_{a_p}(n), W_{b_1}(n), \ldots, W_{b_q}(n))$, for all $n\geq 0$.

The proof of Theorem~\ref{th:two_paths} uses the ODE method, as for the  cone graph in the previous section. We roughly follow the same steps here, but there are some important differences.  
We first show in Subsection~\ref{subsec.2paths.1} that almost surely the limiting set $L(\hat{\bf W})$ of the sequence of normalised weights $\hat{{\bf W}}(n):={\bf W}(n)/(n+1)$, 
is contained in $(0,1]^{p+q}$. Then, in Subsection~\ref{subsec.2paths.2}, we give an explicit expression for the vector field $F$ appearing in the stochastic approximation satisfied by $(\hat{{\bf W}}(n))_{n\ge 0}$.  
{In this section, we also define a sequence  of compact sets $(K_n)_{n\ge 0}$ in $[0,1]^{p+q}$, which we prove to be decreasing. }
Then in Subsection~\ref{subsec.2paths.3}, we show that $L(\hat{\bf W})\subseteq K_n$, for all $n\ge 0$, and finally in Subsection~\ref{subsec.2paths.4}, we prove that the intersection of the $K_n$'s is reduced to a single point, which is precisely the limiting point arising in the statement of Theorem~\ref{th:two_paths}.

\subsection{Proof that none of the edges has a limiting weight equal to zero}\label{subsec.2paths.1}
We prove here the following result. 

\begin{proposition}\label{prop:multiedges}
Let $p$ and $q$ be integers such that $\min(p, q)\geq 2$.
If $\mathcal G = (V,E)$ is the $(p,q)$-path graph, then there exists a constant $c_{p,q}>0$, such that almost surely for all $e\in E$,
\[\liminf_{n\to +\infty} \frac{W_e(n)}n> c_{p,q}.\]
\end{proposition}
This proposition is proved by induction on $k$ for $a_k\in \{a_1, \ldots, a_p\}$ and by induction on $\ell$ for $b_\ell\in \{b_1, \ldots, b_q\}$. We prove the base case separately in the following lemma.
\begin{lemma}\label{lem:liminf_a1}
In the $(p+q)$-path graph, if $q\geq 2$, then there exists $c>0$, such that almost surely, $\liminf_{n\to+\infty} W_{a_1}(n)/n\ge c$, and by symmetry, if $p\geq 2$, then  almost surely, $\liminf_{n\to+\infty} W_{b_1}(n)/n\ge c$.
\end{lemma}

\begin{proof}
For all $n\geq 0$, we have
\begin{equation}\label{eq:prob_a1}
\mathbb P(W_{a_1}(n+1) =  W_{a_1}(n)+1 \mid {\bf W}(n))
=\mathbb P(a_1 \in \gamma(n+1)\mid {\bf W}(n))
= \frac{W_{a_1}(n)}{W_{a_1}(n)+
\frac{1}{\sum_{i=1}^q W_{b_i}(n)^{-1}}}.
\end{equation}
Note that, by definition of the model, 
for all $n\geq 0$, 
we have either $a_p\in\gamma(n+1)$ or $b_q\in\gamma(n+1)$ but not both.
This implies that, for all $n\geq 0$,
\[W_{a_p}(n)+W_{b_q}(n)=n+2.\]
By definition of the model again, 
we have that $a_p\in\gamma(n+1)\Rightarrow a_{p-1}\in \gamma(n+1) \Rightarrow \cdots \Rightarrow a_1\in\gamma(n+1)$, and 
$b_q\in\gamma(n+1)\Rightarrow b_{q-1}\in\gamma(n+1)\Rightarrow  \cdots \Rightarrow  b_1\in\gamma(n+1)$, which imply that, for all $n\geq 0$,
\[
W_{a_p}(n)\leq W_{a_{p-1}}(n)\leq \cdots \leq W_{a_1}(n)\leq n+1\quad\text{ and }\quad 
W_{b_q}(n)\leq W_{b_{q-1}}(n)\leq\cdots \leq  W_{b_1}(n)\leq n+1.
\]
Thus,~\eqref{eq:prob_a1} implies that
\[\mathbb P(W_{a_1}(n+1) =  W_{a_1}(n)+1 \mid {\bf W}(n))
\geq \frac{W_{a_1}(n)}{W_{a_1}(n)+\frac{1}{q(n+1)^{-1}}}
=\frac{Z(n)}{Z(n)+\nicefrac 1q},
\]
where we have set $Z(n):=W_{a_1}(n)/(n+1)$, for all $n\geq 0$.
This implies that $(W_{a_1}(n))_{n\ge 0}$ stochastically dominates a $G$-urn process, 
where, for all $z\in[0,1]$,
\[G(z) = \frac{z}{z+\nicefrac1{q}}.\]
Note that $G(z) = z$ if and only if $z=0$ or $z=1-{\nicefrac1q}$
(which is positive because $q\geq 2$, by assumption);
furthermore, one can check that $G$ is $C^1$ and $G'(0) = q > 1$, 
implying that the normalized $G$-urn process converges almost surely to $c:= 1-\nicefrac1q>0$, when $n\to+\infty$, by Proposition~\ref{lem.urne}.
This concludes the proof.
\end{proof}

\begin{proof}[Proof of Proposition~\ref{prop:multiedges}]
We prove the result by induction on $k\in \{1, \ldots, p\}$. 
The case $k=1$ follows from Lemma~\ref{lem:liminf_a1}.
Assume the result holds true for some $k<p$, i.e.\
there exists {$c\in (0,1/2)$} such that almost surely $\liminf W_{a_k}(n)/n> c$. 
{Let us define the events 
\[E_m:=\{W_{a_k}(n) \ge c n, \forall n\ge m\} \qquad (m\geq 1).\]
Then, $\bigcup_m E_m$ holds almost surely.}

Let $\tau_n$ be the time 
when edge~$a_k$ is reinforced for the $n$-th time 
(so by definition $W_{a_k}(\tau_n) = n+1$). 
Note that by definition $\tau_n\ge n$ for all $n\ge 1$. Thus, for all $n\ge m$, on $E_m$, using that {$W_{a_{k-1}}(\tau_n)\le \tau_n+1$ and $\tau_n\le (n+1)/c$, one has
\[\frac{W_{a_k}(\tau_n)}{W_{a_k}(\tau_n) + W_{a_{k-1}}(\tau_n)} \ge \frac{c(n+1)}{c(n+1)+n+1+c} \ge c/2.\]  }
It follows that, almost surely on the event $E_m$, if the weight of $a_{k+1}$ is zero, then,
for all $n\ge m$, the $\tau_n$-th ant makes at least a geometric number 
of crossings of edge $a_k$, with success probability $\nu:=1-c/2$, 
before jumping across edge $a_{k-1}$. 
Consequently, on $E_m$, and for $n\ge m$, 
the probability to reinforce edge $a_{k+1}$ at time $\tau_n$ 
is at least $1- (1-\rho(n))^{X_n}$, 
where $(X_n)_{n\ge 0}$ is a sequence of i.i.d.\ geometric random variables with parameter~$\nu$, 
and 
\[\rho(n) = \frac{W_{a_{k+1}}(\tau_n)}{W_{a_{k+1}}(\tau_n)+ n+ 1}.\]
Thus, letting ${u}_{k+1}(n) = W_{a_{k+1}}(\tau_n)/(n+1)$, we find 
\ba
\mathbb E[W_{a_{k+1}}(\tau_{n+1}) -W_{a_{k+1}}(\tau_n)\mid \mathcal F_{\tau_n}]  
&\geq \mathbb E[1-(1-\rho(n))^{X_n}\mid W_{a_{k+1}}(\tau_n)]
= 1-\frac{\nu(1-\rho(n))}{1-(1-\nu)(1-\rho(n))}\\
&= \frac{\rho(n)}{1-  (1-\nu)(1-\rho(n))}
= \frac{{u}_{k+1}(n)}{{u}_{k+1}(n) + \nu},
\ea
with $(\mathcal F_i)_{i\ge 0}$ the natural filtration of the process,
and where we used that $\mathbb E x^{X_n} = \nu x/(1-(1-\nu)x)$ for all $x\in(0,1)$. 
It follows that on the event $E_m$, the process $(W_{a_{k+1}}(\tau_n))_{n\ge m}$ stochastically dominates a $G$-urn process (starting from $W_{a_{k+1}}(\tau_m)$), 
with $G(x): = \frac{x}{x+\nu}$. 
Since $G(x)> x$ for all $x\in(0,c/2)$, it follows from Proposition~\ref{lem.urne} that almost surely 
$\liminf {u}_{k+1}(n) \ge c/2$. We deduce the induction step, using that by hypothesis, $\limsup \tau_n/n \le 1/c$. 
\end{proof}

\subsection{The stochastic algorithm and 
a sequence of decreasing compact subspaces.}\label{subsec.2paths.2} 
Recall that we set $\hat {\bf W}(n) := {\bf W}(n)/(n+1)$, for all $n\geq 0$.
Note also that, by definition of the model, for all $n\geq 0$,
\[\hat {{\bf W}}(n) \in \mathcal E' := 
\{\bs w \in \mathcal E \colon w_{a_p} = 1 - w_{b_q}, w_{b_q}\leq w_{b_{q-1}}\leq\cdots \leq w_{b_1},
w_{a_p}\leq w_{a_{p-1}}\leq\cdots\leq w_{a_1}\},\]
with $\mathcal E$ as defined in \eqref{def.E}. 
Moreover, for all $n\geq 0$, we have
\[\hat {\bf W}(n+1) = \hat{\bf W}(n) + 
\frac1{n+2}\big(F(\hat{\bf W}(n)) + \xi_{n+1}\big),
\]
with $\xi_{n+1}$ some martingale difference and with $F$ as defined in~\eqref{def.F}. More specifically the coordinates of $F$ can be computed explicitly here, 
and are given by, for all $1\leq k\leq p$ and $1\leq \ell\leq q$, for all $\bs w\in \mathcal E'$,
\begin{equation}\label{eq:Fab}
F_{a_k}(\bs w) = \frac{S^a_k(\bs w)}
{S^a_k(\bs w)+S^b_q(\bs w)}
-w_{a_k}
\quad\text{ and }\quad
F_{b_\ell}(\bs w) = \frac{S^b_\ell(\bs w)}
{S^b_\ell(\bs w)+S^a_p(\bs w)}
-w_{b_k},
\end{equation}
where we have defined, for any $\bs w \in [0,1]^E$, $m\in\{a,b\}$, 
and $s$ an integer such that $1\leq s\leq p$ if $m = a$, and $1\leq s\leq q$ if $m = b$,
\begin{equation}\label{eq:def_Sq}
S^m_s(\bs w) = \frac1{\sum_{i=1}^s\frac1{w_{m_i}}}.
\end{equation}
Note that, for all $1\le k\le p$, $S_k^a(\bs w)=0$ if and only if $w_{a_i}=0$ for some $1\le i \le k$, 
and for all $1\le \ell\le q$, $S_\ell^b(\bs w)=0$ if and only if $w_{b_i}=0$ for some $1\le i \le \ell$. 

\bigskip
To prove Theorem~\ref{th:two_paths}, we use the ODE method and thus start by studying the solutions of the equation ${\bs{\dot y}} =F({\bs y})$. To do so, we define a sequence $(K_n)_{n\geq 0}$ of decreasing compact subsets of $\mathcal E'$ such that (A) for all $n\geq 0$, $L(\hat{\bf W})\subseteq K_n$, and (B) the intersection of all these compacts is $\{\bs w^*\}$, where $w_{a_k}^* = \alpha^k$ and $w_{b_\ell}^* = \beta^\ell$, with $(\alpha, \beta)$ as in Theorem ~\ref{th:two_paths}.
We prove (A) in Section~\eqref{subsec.2paths.3}, and (B) in Section~\ref{subsec.2paths.4}.
In the rest of this section, we define the sequence $(K_n)_{n\geq 0}$ and show that it is decreasing, 
i.e.\ that $K_{n+1} \subset K_n$ for all $n\ge 0$.
For this we need some additional notation.  
For all $\bs u,\bs v \in \mathcal E'$, we let
\[H_{a_k}(\bs u,\bs v)
= \frac{S^a_k(\bs u)}{S^a_k(\bs u)+S^b_q(\bs v)}
\quad\text{ and }\quad
H_{b_\ell}(\bs u,\bs v)
= \frac{S^b_\ell(\bs u)}{S^b_\ell(\bs u)+S^a_p(\bs v)},
\]
for all $1\leq k\leq p$ and $1\leq \ell\leq q$.
Recall further that by Proposition~\ref{prop:multiedges}, there exists a constant $c_{p,q}>0$, such that almost surely 
\begin{equation}\label{K0def}
L(\hat{\bf W}) \subseteq \{\bs w\colon w_e\ge c_{p,q}, \text{ for all }e\in E\}.
\end{equation}
{We also define $\bs w^*$ as the limiting vector appearing in the statement of Theorem~\ref{th:two_paths}. More precisely, we have }
\begin{equation}\label{w*}
{w_{a_k}^* = \alpha^k, \quad \text{and} \quad w_{b_\ell}^* = \beta^\ell,}
\end{equation}
{for all $1\le k\le p$, $1\le \ell \le q$, with $(\alpha,\beta)$ the unique solution of the system~\eqref{eq:system} in $(0,1)^2$ (existence and {uniqueness} of the solution for this system of equations will be proved later, at the end Subsection~\ref{subsec.2paths.4}).}
We then define $\bs u^{(0)}$ and $\bs v^{(0)}$ by 
\[u^{\sss (0)}_{a_k}=u_0^k, \quad 
u^{\sss (0)}_{b_\ell}=u_0^\ell,\quad \text{ and }\quad  
v^{\sss (0)}_{a_k}=v^{\sss (0)}_{b_\ell}=1,\] 
for all $1\le k\le p$, $1\le \ell \le q$, with $u_0$ chosen arbitrarily so that 
\[0< u_0<\min(1-\nicefrac1q,\alpha,\beta,c_{p,q}).\] 
The fact that we choose $u_0\le \min(\alpha,\beta)$ entails in particular 
\[
0< u^{\sss (0)}_e\leq  w^*_e, \quad \text{for all }e\in E.
\]
Next we define inductively two sequences $(\bs u^{\sss (n)})_{n\ge 0}$ and $(\bs v^{\sss (n)}))_{n\ge 0}$, by 
\[\bs u^{\sss (n+1)}= H(\bs u^{\sss (n)}, \bs v^{\sss (n)}), \quad \text{and}\quad \bs v^{\sss (n+1)}=  H(\bs v^{\sss (n)}, \bs u^{\sss (n)}).\]
Finally, we define the sequence $(K_n)_{n\ge 0}$ by 
\[
K_n:=\{\bs w\in \mathcal E'\colon u^{\sss (n)}_e\le w_e \le v^{\sss (n)}_e,  \text{ for all }e\in E\}  \quad (\forall n\ge 0).
\]
We prove now that this sequence is decreasing, that is, $K_{n+1}\subset K_n$, for all $n\ge 0$. More precisely, we prove that, for all $n\ge 0$, 
\ban
\bs u^{\sss (n)}&<H(\bs u^{\sss (n)}, \bs v^{\sss (n)})=\bs u^{\sss (n+1)}\label{ineq1}\\
\bs v^{\sss (n)}&>H(\bs v^{\sss (n)}, \bs u^{\sss (n)})=\bs v^{\sss (n+1)}.\label{ineq2}
\ean
We reason by induction on $n$: first note that, for all $1\le k\le p$ and $1\le \ell \le q$,
\[
S^a_k(\bs u^{\sss (0)})=u_0^k\frac{1-u_0}{1-u_0^k}\quad\text{ and }\quad S^b_\ell(\bs v^{\sss (0)})=\frac{1}{\ell},
\]
which implies, using that $1-u_0> 1/q$,
\[u^{\sss (1)}_{a_k}=H_{a_k}(\bs u^{\sss (0)},\bs v^{\sss (0)})=\frac{u_0^k(1-u_0)}{u_0^k(1-u_0)+(1-u_0^k)/q}> \frac{u_0^k/q}{u_0^k/q+(1-u_0^k)/q}=u_0^k=u^{\sss (0)}_{a_k},\]
and
\[v^{(1)}_{a_k}=H_{a_k}(\bs v^{\sss (0)},\bs u^{\sss (0)})< 1=v^{\sss (0)}_{a_k}.\]
Similarly, for all $1\le \ell \le q$, 
\[u^{\sss (1)}_{b_\ell}> u^{\sss (0)}_{b_\ell}\quad \text{ and }\quad v^{\sss (1)}_{b_\ell}< v^{\sss (0)}_{b_\ell}.\] 
We now proceed to the induction step and assume that, for some $n\ge1$ $u^{\sss (n)}_{e}> u^{\sss (n-1)}_{e}$ and $v^{\sss (n)}_{e}< v^{\sss (n-1)}_{e}$ for all $e\in E$. 
We then simply observe that, for all $e\in E$, 
\[u_e^{\sss (n+1)}=H_e(\bs u^{\sss (n)},\bs v^{\sss (n)})> H_e(\bs u^{\sss (n-1)},\bs v^{\sss (n-1)})=u_e^{\sss (n)},\] 
and 
\[v_e^{\sss (n+1)}=H_e(\bs v^{\sss (n)},\bs u^{\sss (n)})< H_e(\bs v^{\sss (n-1)},\bs u^{\sss (n-1)})=v_e^{\sss (n)},\] 
which proves the induction step, and thus concludes the proofs of \eqref{ineq1} and \eqref{ineq2}. In other words we just have proved that $(K_n)_{n\ge 0}$ is indeed a sequence of decreasing sets.

\subsection{Proof that $L(\hat{\bf W})\subseteq K_n$, for all $n\ge 0$.} \label{subsec.2paths.3}
We use an induction argument. Note first that by~\eqref{K0def} and the definition of $K_0$, one has almost surely $L(\hat{\bf W}) \subseteq K_0$, using also the hypothesis $u_0\le c_{p,q}$. 
We now prove the induction step, i.e. that almost surely, if $L(\hat{\bf W}) \subseteq K_n$, for some $n\ge 0$, then also 
$L(\hat{\bf W}) \subseteq K_{n+1}$.

To do so, we first look at $F_{a_k}(\bs w)$ for $1\leq k\leq p$ 
and $\bs w$ such that $u_e\leq w_e\leq v_e$ for all $e\in E$: we have
\[F_{a_k}(\bs w) 
= \frac{S^a_k(\bs w)}
{S^a_k(\bs w)+S^b_q(\bs w)}
-w_{a_k}
\ge \frac{S^a_k(\bs u)}
{S^a_k(\bs u)+S^b_q(\bs v)}- w_{a_k}=H_{a_k}(\bs u,\bs v)- w_{a_k}.
\]
Thus, if  $u_e\leq w_e\leq v_e$ for all $e\in E$ and
\[w_{a_k}< H_{a_k}(\bs u,\bs v),
\]
then $F_{a_k}(\bs w)>0$. Also, for all $1\leq k\leq p$, if $u_e\leq w_e\leq v_e$ for all $e\in E$ and
\[w_{a_k}>H_{a_k}(\bs v,\bs u),
\]
then $F_{a_k}(\bs w)<0$.
The same argument leads to
\ba
w_{b_\ell}<H_{b_\ell}(\bs u,\bs v)\quad &\Longrightarrow \quad F_{b_\ell}(\bs w)>0;\\
w_{b_\ell}>H_{b_\ell}(\bs v,\bs u)\quad &\Longrightarrow \quad F_{b_\ell}(\bs w)<0,
\ea
for all $1\leq \ell\leq q$, and if $u_e\leq w_e\leq v_e$ for all $e\in E$.
These facts imply that, for all $n\ge 0$, and for any $\bs w\in K_n$, the flow of the ODE ${\bs{\dot y}} =F({\bs y})$ started at $\bs w$ converges to $K_{n+1}$ and the convergence is uniform on $K_n$ by a similar argument to step (4) in Section~\ref{sec:cornet}. 
Therefore if we already know that $L(\hat{\bf W}) \subseteq K_n$, then it means that $L(\hat{\bf W}) \cap K_{n+1}$ is an attractor of the flow restricted to $L(\hat{\bf W})$. 
Thus by Theorem~\ref{th:pemantle}, we deduce that $L(\hat{\bf W}) \subseteq K_{n+1}$.

Altogether this proves that $L(\hat{\bf W})$ is almost surely included in the intersection of all the $K_n$'s.

\subsection{Identification of the intersection of the $K_n$'s.} \label{subsec.2paths.4} 
We show here that the intersection of the $K_n$'s is reduced to the single point $\bs w^*$, which appears as the limiting vector in the statement of Theorem~\ref{th:two_paths} (see also \eqref{w*} above).

Since the sequences $(u_e^{\sss (n)})$ and $(v_e^{\sss (n)})$ are all {monotonic} and bounded, 
they all converge. 
We let $\bs u^* = \lim_{n\to+\infty} \bs u^{\sss (n)}$ and $\bs v^*= \lim_{n\to+\infty} \bs u^{\sss (n)}$. 
Because $H$ is continuous on $\mathcal E'$, 
we have 
\[\bs u^*= H(\bs u^*,\bs v^*),\quad \text{and}\quad \bs v^*=H(\bs v^*,\bs u^*).\]
The equation $\bs u^*=H(\bs u^*,\bs v^*)$ can be written as, for all $1\leq k\leq p$, $1\leq\ell\leq q$,
\[u^*_{a_k} = \frac{S^a_k(\bs u^*)}{S^a_k(\bs u^*)+S^b_q(\bs v^*)}
\quad\text{ and }\quad
u^*_{b_\ell} = \frac{S^b_\ell(\bs u^*)}{S^b_\ell(\bs u^*)+S^a_p(\bs v^*)}.\]
Using that $S^a_1(\bs u^*)=u^*_{a_1}$ and $S^b_1(\bs u^*)=u^*_{b_1}$, this implies that 
\begin{equation}\label{eq:alpha_beta}
u^*_{a_1} = 1-S^{b}_q(\bs v^*)=:\alpha
\quad\text{ and }\quad 
u^*_{b_1} = 1-S^{a}_p(\bs v^*)=:\beta,
\end{equation}
and, for all $2\leq k\leq p$, $2\leq \ell\leq q$,
\[
u^*_{a_k} = \frac{S^a_k(\bs u^*)}{S^a_k(\bs u^*)+1-u^*_{a_1}}
\quad\text{ and }\quad
u^*_{b_\ell} = \frac{S^b_\ell(\bs u^*)}{S^b_\ell(\bs u^*)+1-u^*_{b_1}}.
\]
We first show by induction that this implies $u^*_{a_k} = \alpha^k$ 
and $u^*_{b_\ell} = \beta^\ell$ for all $1\leq k\leq p$ and $1\leq \ell\leq q$.
Indeed, if for some $1< k\le p$, and all $1\leq i< k$, $u^*_{a_i} = \alpha^i$, then
\[
u^*_{a_k} = \frac{\left( \frac{1}{u^*_{a_k}}+\frac{1-\alpha^{k-1}}{\alpha^{k-1}(1-\alpha)}\right)^{-1}}{\left( \frac{1}{u^*_{a_k}}+\frac{1-\alpha^{k-1}}{\alpha^{k-1}(1-\alpha)}\right)^{-1}+1-\alpha} 
= \frac{1}{1+(1-\alpha)\left( \frac{1}{u^*_{a_k}}+\frac{1-\alpha^{k-1}}{\alpha^{k-1}(1-\alpha)}\right)}, 
\]
and a straightforward calculation yields that $u^*_{a_k}=\alpha^k$,
as claimed.
The proof of $u^*_{b_\ell} = \beta^\ell$ for all $1\leq\ell\leq q$ is similar.

Since $(\bs u^*,\bs v^*)$ also satisfies the symmetric equation $\bs v^* = H(\bs v^*, \bs u^*)$, 
we get that
\begin{equation}\label{eq:bar_alpha_beta}
v^*_{a_1} = 1-S^{b}_q(\bs u^*)=:\bar\alpha
\quad\text{ and }\quad 
v^*_{b_1} = 1-S^{ a}_p(\bs u^*)=:\bar\beta,
\end{equation}
and $v^*_{a_k} = \bar\alpha^k$, 
$v^*_{b_\ell} = \bar\beta^\ell$ for all $1\leq k\leq p$ and $1\leq\ell\leq q$.
Using this into~\eqref{eq:alpha_beta}, 
and using the definition of $S_p$ and $S_q$ (see~\eqref{eq:def_Sq}), we get
\[\alpha = 1-\frac1{\sum_{i=1}^q \bar\beta^{-i}}
\quad\text{ and }\quad
\beta = 1-\frac1{\sum_{i=1}^p \bar\alpha^{-i}}.
\]
Similarly, using the fact that $u^*_{a_k} = \alpha^k$ and $u^*_{b_\ell} = \beta^\ell$ for all $1\leq k\leq p$ and $1\leq\ell\leq q$, together with Equation~\eqref{eq:bar_alpha_beta}, we get
\[\bar\alpha = 1-\frac1{\sum_{i=1}^q \beta^{-i}}
\quad\text{ and }\quad
\bar\beta = 1-\frac1{\sum_{i=1}^p \alpha^{-i}}.\]
Note that
\begin{equation}\label{eq:first_equiv}
\alpha = 1-\frac1{\sum_{i=1}^q \bar\beta^{-i}}\quad\Leftrightarrow\quad
1-\alpha = \frac{\bar\beta^q(1-\bar\beta)}{1-\bar\beta^q},
\end{equation}
and, similarly,
\begin{equation}\label{eq:sec_equiv}
\bar\beta = 1-\frac1{\sum_{i=1}^p \alpha^{-i}}\quad\Leftrightarrow\quad
1-\bar\beta =  \frac{\alpha^p(1-\alpha)}{1-\alpha^p}.
\end{equation}
For all integer $p\geq 2$ and $x\in [0,1]$, we let 
\[f_p(x) = 1-\frac1{\sum_{i=1}^p x^{-i}} = 1-\frac{x^p(1-x)}{1-x^p}.\]
With this notation, we have $\alpha = f_q(\bar\beta)$ and $\bar\beta = f_p(\alpha)$ (and similarly for $\bar\alpha$ and $\beta$), and thus
\[\alpha = f_q\circ f_p(\alpha) \quad\text{ and } \quad \bar\alpha = f_q\circ f_p(\bar\alpha).\]
We now show that $f_p$ is a contraction for all $p\geq 2$, 
implying that $f_p\circ f_q$ is also a contraction, and thus admits a unique fixed point, 
which implies $\alpha = \bar\alpha$ (and thus $\beta = \bar\beta$).

\begin{lemma}\label{lem:contraction}
For all $p\geq 2$, the function $f_p\colon [0,1]\to \mathbb R$ defined by 
\[f_p(x) := 1-\frac{x^p(1-x)}{1-x^p},\]
is a contraction.
\end{lemma}

\begin{proof}
First note that $f_p$ can be extended to a continuous function on $[0,1]$ by setting $f_p(1) = 1-\nicefrac1p$. To prove that $f_p$ is a contraction, we show that there exists $\varepsilon>0$ such that, for all $x\in[0,1]$, $|f'_p(x)|\leq 1-\varepsilon$. 
First note that, for all $x\in [0,1)$,
\[f'_p(x) = -\frac{x^{p-1}(x^{p+1}-(p+1)x+{p})}{(1-x^p)^2}.\]
For all $\varepsilon>0$, we have that 
\begin{linenomath}\begin{align*}
f'_p(x)\leq 1-\varepsilon
&\;\Leftrightarrow\; x^{2p}-(p+1)x^p +px^{p-1}
\leq 1-\varepsilon -2(1-\varepsilon)x^p+(1-\varepsilon) x^{2p}\\
&\;\Leftrightarrow\; 0\leq 1-\varepsilon - px^{p-1} + (p-1+2\varepsilon)x^p-\varepsilon x^{2p}
=:\varphi(x).
\end{align*}
\end{linenomath}
To understand $\varphi(x)$ on $[0,1]$, we look at its derivative: for all $x\in [0,1]$,
\[\varphi'(x) 
= -p(p-1)x^{p-2}+p(p-1+2\varepsilon)x^{p-1}-2p\varepsilon x^{2p-1} = x^{p-2}\psi(x),\]
where
\[\psi(x) = -p(p-1)+p(p-1+2\varepsilon)x-2p\varepsilon x^{p+1}.\]
Note that $\psi'(x) = p(p-1+2\varepsilon)-2p(p+1)\varepsilon x^p$ is non-negative if and only if
\[x^p\leq \frac{p(p-1+2\varepsilon)}{2p(p+1)\varepsilon}.\]
For all $\varepsilon$ small enough, the right-hand side of this inequality is larger than one (because $p\geq 2$), implying that for such $\varepsilon$, $\psi'(x)$ is non-negative and thus $\psi$ is non-decreasing on $[0,1]$.
And thus, for all $x\in [0,1]$, $\psi(x)\leq \psi(1) = 0$.
Therefore, since $\varphi'(x) = x^{p-2}\psi(x)$, we get that $\varphi'(x)\leq 0$ for all $x\in [0,1]$,
and thus $\varphi$ is non-increasing on $[0,1]$. This implies that
$\varphi(x)\leq \varphi(0) = 1-\varepsilon$, and thus concludes the proof.
\end{proof}

We have thus proved that $\alpha = \bar\alpha$ and $\beta = \bar\beta$, 
where we recall that $(\alpha, \beta)$ is the unique solution of $\alpha = f_q(\beta)$ and $\beta = f_p(\alpha)$ in $(0,1)^2$, i.e.
\begin{equation}\label{eq:system2}
\begin{cases}
\alpha = 1-\frac{\beta^q(1-\beta)}{1-\beta^q}&\\
\beta = 1-\frac{\alpha^p(1-\alpha)}{1-\alpha^p}.&
\end{cases}
\end{equation}
It only remains to show that $(\alpha,\beta)$ is also {a} solution of~\eqref{eq:system},
and that it is the unique solution of~\eqref{eq:system} on $(0,1)^2$.
Since $(\alpha,\beta)$ is {a} solution of~\eqref{eq:system2}, we get that
\[(1-\alpha)(1-\beta^q) =\beta^q(1-\beta)
\;\Rightarrow\;  1-\alpha=\beta^q(2-\alpha-\beta),\]
and, similarly,
\[(1-\beta)(1-\alpha^p) =\alpha^p(1-\alpha)
\;\Rightarrow\;  1-\beta=\alpha^p(2-\alpha-\beta).\]
This implies 
\[2-\alpha-\beta = \frac{1-\beta}{\alpha^p} = \frac{1-\alpha}{\beta^q},\]
and thus $(\alpha,\beta)$ satisfies the second equation of~\eqref{eq:system}.
Furthermore,
\[\alpha^p+\beta^q = \frac{1-\beta}{2-\alpha=\beta} + \frac{1-\alpha}{2-\alpha-\beta} = 1,\]
implying that $(\alpha,\beta)$ is solution of~\eqref{eq:system}.

To prove that~\eqref{eq:system} has a unique solution on $(0,1)^2$, we show that any solution of~\eqref{eq:system} on $(0,1)^2$ is also a solution of~\eqref{eq:system2} (since the latter has a unique solution on $(0,1)^2$, this concludes the proof).
Indeed, if $(\alpha,\beta)\in (0,1)^2$ is a solution of~\eqref{eq:system}, then
\[1-\alpha = \frac{\beta^q(1-\beta)}{\alpha^p} = \frac{\beta^q(1-\beta)}{1-\beta^q},
\]
which implies the first equation of~\eqref{eq:system2}. The second equation of~\eqref{eq:system2} can obtained similarly. 
Therefore, if $(\alpha, \beta)$ is a solution of~\eqref{eq:system}, then it is also a solution of~\eqref{eq:system2}, which concludes the proof of Theorem~\ref{th:two_paths}.


\section{The lozenge: proof of Proposition~\ref{prop:lozenge}}
\label{sec:lozenge}
First recall that, by Proposition~\ref{prop.stoc.approx} and the specificities of the lozenge graph, if we let $\hat {\bf W}(n) = \frac{{\bf W}(n)}{n+2}$ ($\forall n\geq 0$), 
then, for all $n\geq 0$, $\hat{\bf W}(n)\in\mathcal E'$, where
\[\mathcal E' :=
\{\bs w=(w_1, w_2, w_3, w_4, w_5)\in \mathcal E \colon
w_2 + w_5 = 1, w_1+w_4\geq 1, w_2 \leq w_1+w_3, w_5\leq w_3+w_4\},\]
with $\mathcal E$ as defined in~\eqref{def.E}. 
The first condition ($w_2+w_5=1$) is satisfied by $\hat{\bf W}(n)$ because, by definition of the model, each ant reinforces either edge 2 or edge 5 but not both. The second condition ($w_1+w_4\geq 1$) 
is redundant with the fact that $\bs w \in \mathcal E$ (it is the same condition as $\pi_{\bs w}(N)\ge 1$). The third condition ($w_2 \leq w_1+w_3$) holds because each ant that reinforces edge~2 also reinforces either edge~1 or edge~3. The fourth condition is the symmetric of the third one.

{\begin{remark}
As in Section~\ref{sec:cornet}, 
it is convenient to define $\hat{\bf W}(n)$ as ${\bf W}(n)/(n+2)$ rather than as ${\bf W}(n)/(n+1)$. This is because with the former definition, we have $\hat{W}_2(n) + \hat{W}_5(n) = 1$ for all $n\geq 0$, which is not true with the latter definition (although it holds asymptotically as $n\to+\infty$).
\end{remark}}

Moreover, for all $n\geq 0$,
\[\hat {\bf W}(n+1) =\hat  {\bf W}(n) + \frac1{n+3}\big(F(\hat{\bf W}(n)) + \xi_{n+1}\big),\]
with $\xi_{n+1}$ some martingale difference, and where
$F_i(w) = p_i(w) - w_i$, with $p_i(w) = \mathbb P(e_i\in \gamma_{n+1}\mid \hat{\bf W}(n) = w)$ (note that this probability does not depend on~$n$).

The first step in the proof of Proposition~\ref{prop:lozenge} is to compute these probabilities $p_i(\bs w)$, for $1\le i \le 5$. 
A straightforward calculation, which we carry out in Section~\ref{sec:app}, shows that 
for all $\bs w\in \mathcal E'$,
\begin{linenomath}\begin{align}
p_1(\bs w) &= \frac{w_1}{w_1+w_4} + \frac{w_4}{w_1+w_4}\cdot
\frac{\frac{w_1}{w_3+w_4+w_5}\left(\frac{w_4}{w_1+w_4}+\frac{w_3}{w_1+w_2+w_3}\right)}
{1-\frac{w_4^2}{(w_1+w_4)(w_3+w_4+w_5)}-\frac{w_3^2}{(w_1+w_2+w_3)(w_3+w_4+w_5)}}\notag\\
p_2(\bs w)&= \frac{w_2(w_1(w_3+w_4+w_5)+w_3w_4)}{(w_1+w_4)(w_3+w_2w_5+\frac{w_1w_4}{w_1+w_4})}\notag\\
p_3(\bs w) &= \frac{w_3\left(\frac{w_1}{w_1+w_2+w_3}+\frac{w_4}{w_3+w_4+w_5}\right)}
{w_1+w_4 - \left(\frac{w_1^2}{w_1+w_2+w_3}+\frac{w_4^2}{w_3+w_4+w_5}\right)}.
\label{eq:formule_lozenge}
\end{align}
\end{linenomath}
By symmetry, we also have $p_4(\bs w) = p_1(w_4, w_5, w_3, w_1, w_2)$ 
and $p_5(\bs w) = p_2(w_4, w_5, w_3, w_1, w_2)$.
Furthermore, since, by definition of the model, each ant reinforces either edge 2, or edge 5, but not both, we have $p_2(\bs w) = 1-p_5(\bs w)$.
Note also that, for all $\bs w\in\mathcal E'$,
\begin{equation}\label{F2.lozenge}
F_2(\bs w) 
= p_2(\bs w) - w_2 
=  \frac{w_2w_5 (\frac{w_1}{w_1+w_4} - w_2)}{w_3+w_2w_5+\frac{w_1w_4}{w_1+w_4}},
\end{equation}
and 
\begin{equation}\label{F3.lozenge}
p_3(\bs w)  = \frac{w_3(\alpha + \beta)}{\alpha(w_3+w_2) + \beta(w_3+w_5)},
\end{equation}
with 
\begin{equation}\label{alpha.beta}
{\alpha}:= \frac{w_1}{w_1+w_2+w_3}, \qquad \text{and}\qquad {\beta} := \frac{w_4}{w_3+w_4+w_5},
\end{equation}
where by convention we set ${\alpha} = 0$ when $w_1=0$, and similarly ${\beta}=0$, when $w_4=0$. 

\bigskip
The second step is the following fact. 
\begin{lemma}\label{lem:liminf_lozenge}
Almost surely $\liminf_{n\to +\infty} \frac{W_1(n)}{n}>0$, and by symmetry $\liminf_{n\to +\infty} \frac{W_4(n)}{n}>0$ {almost surely}.
\end{lemma}
\begin{proof}
Note that for all $\bs w\in \mathcal E'$, 
\[\frac{p_1(\bs w)}{w_1} \ge \frac{1}{w_1+w_4}  + \frac{w_4^2}{w_1+w_4} \cdot
\frac{1}{(w_3+w_4+w_5)(w_1+w_4) -w_4^2 }.\]
When $w_1\to 0$, we have $w_4\to 1$ because $1-w_1\le w_4\le 1$ for all $\bs w\in \mathcal E'$.
Using in addition the fact that $w_3 + w_5\le 2$ for all $\bs w\in \mathcal E'$, we get that
\[\liminf_{w_1\to 0} \frac{p_1(\bs w)}{w_1} \ge \frac 32,\]
and then the result follows from Corollary~\ref{prop.urne}.  
\end{proof}

We next prove the following result. 
\begin{lemma}\label{sign.F}
For all $\bs w\in \mathcal E'$, one has 
\begin{itemize}
\item[$(i)$] If $w_2 < \frac{w_1}{w_1+w_4}$, then $F_2(\bs w) > 0$, and if $w_2 > \frac{w_1}{w_1+w_4}$, then $F_2(\bs w) < 0$. 
\item[($ii)$] If $w_2\ge  \frac{w_1}{w_1+w_4}$, and $0<w_1<w_4$, then $F_1(\bs w)w_4 - F_4(\bs w)w_1 >0$. Likewise, if $w_2\le  \frac{w_1}{w_1+w_4}$, and $w_4<w_1<1$, then $F_1(\bs w)w_4 - F_4(\bs w)w_1 <0$. 

\end{itemize}
\end{lemma}
\begin{proof}
The first claim follows directly from~\eqref{F2.lozenge}. For the second claim, note that if $w_2\ge \frac{w_1}{w_1+w_4}$, then with the notation of~\eqref{alpha.beta}, one has ${\lambda_{1}} \le \frac{w_1}{w_1+\frac{w_1}{w_1+w_4} + w_3}$, and if in addition $w_1<w_4$, we get 
\[{\lambda_{1}}\le \frac{w_1}{w_1+\frac{w_1}{w_1+w_4} + w_3}< \frac{w_4}{w_4+\frac{w_4}{w_1+w_4} + w_3} \le {\lambda_{4}},\]
since $w_2 \ge \frac{w_1}{w_1+w_4}$ is equivalent to $w_5 \le \frac{w_4}{w_1+w_4}$ (using that for $w\in \mathcal E$, $w_5 = 1-w_2$). 
Then, we get using \eqref{eq:formule_lozenge}, and again ${\lambda_{4}} \ge {\lambda_{1}}$, and $w_4>w_1$, 
\ba
&F_1(\bs w)w_4 - F_4(\bs w)w_1 
= p_1(\bs w)w_4 - p_4(\bs w)w_1 \\
 &= \frac{w_1w_4}{w_1+w_4} \left\{ 
\frac{ \frac{w_4{\lambda_{4}}}{w_1+w_4}+\frac{w_3w_4}{(w_1+w_2+w_3)(w_3+w_4+w_5)}}
{1- \frac{w_4{\lambda_{4}}}{w_1+w_4}-\frac{w_3^2}{(w_1+w_2+w_3)(w_3+w_4+w_5)}} - \frac{ \frac{w_1{\lambda_{1}} }{w_1+w_4}+\frac{w_3w_1}{(w_1+w_2+w_3)(w_3+w_4+w_5)}}
{1- \frac{w_1{\lambda_{1}}}{w_1+w_4}-\frac{w_3^2}{(w_1+w_2+w_3)(w_3+w_4+w_5)}} \right\}
> 0, 
\ea
proving the first statement of (ii). The second statement follows from similar arguments. 
\end{proof}

\begin{proof}[Proof of Proposition \ref{prop:lozenge}]
The proof 	is done in three steps, all three using Corollary~\ref{cor:pemantle}:
\begin{enumerate}[(1)]
\item we first prove that $L(\hat {\bs W})\subset \mathcal H:= \mathcal E'\cap \{w_2={\frac{w_1}{w_1+w_4}}=\nicefrac12\}$,
\item we then infer that $L(\hat {\bs W})\subset \mathcal H':= \mathcal H \cap \{w_3=1/2\}$,
\item and finally conclude that $L(\hat {\bs W})= \{(w^*,  \nicefrac12, \nicefrac12, w^*, \nicefrac12)\}$, where $w^*$ is the unique solution in $[0,1]$ of the equation $2x^3+4x^2-2x-\frac32=0$, as claimed.
\end{enumerate}

(1) We first show that, for all starting point $\bs w\in \mathcal U:=\{\bs w \in \mathcal E' : w_1w_4\neq 0\}$, $\Phi_t(\bs w)$ converges to the set $\mathcal H$. 
Let $(\Phi_t(\bs w)_i)_{i=1,\dots,5}$ denote the coordinates of the vector $\Phi_t(\bs w)$, and let 
\[u(t) = \Phi_t(\bs w)_2 - \frac12, \quad \text{and}\quad 
v(t) = \frac{\Phi_t(\bs w)_1}{\Phi_t(\bs w)_1+ \Phi_t(\bs w)_4} - \frac12.\] 
By definition, taking the derivative along the flow, we get
\[u'(t) = F_2 (\Phi_t(\bs w)), \quad \text{and}\quad 
v'(t) = \frac{F_1(\Phi_t(\bs w))\cdot \Phi_t(\bs w)_4 - F_4(\Phi_t(\bs w))\cdot\Phi_t(\bs w)_1}{(\Phi_t(\bs w)_1 + \Phi_t(\bs w)_4)^2 }.\]
Our aim is to show that $h(t):=\max(|u(t)|,|v(t)|)$  is a Lyapunov function, i.e.~that it is decreasing, for all $t$ smaller than the (possibly infinite) time when it reaches $0$, and that it converges to $0$. To see this, first note that if at some time $t$, one has $0\le v(t)<u(t)$, 
then by Lemma~\ref{sign.F}$(i)$, $h'(t) = u'(t) = F_2(\Phi_t(\bs w)) <0$. By symmetry, if $u(t)<v(t) \le 0$, then $h'(t) = -u'(t) <0$.   
Second, note that, if $v(t) < 0 \le u(t)$, then by Lemma~\ref{sign.F}, we have $u'(t) <0$ and $v'(t)>0$, which entails that $h$ is decreasing in a neighborhood of~$t$ since both its right and left derivatives are negative at this time (it is differentiable if $|u(t)|\neq |v(t)|$). 
Symmetrically, the same holds if $u(t)<0\le  v(t)$. 
Finally when $u(t) = v(t)\neq 0$, we see that $u'(t) = 0$ while 
$v'(t)\neq 0$, so that in a neighborhood of~$t$, 
$u(t)\neq v(t)$, and thus by the previous argument $h$ is again decreasing in a neighborhood of~$t$. 
As a consequence, $h$ is decreasing up to the (possibly infinite) time when it reaches~$0$, and thus converges. Note also that the previous arguments show that~$h$ has a negative right-derivative at any non-zero value, which implies that its only possible limit is zero.
This indeed implies that $\Phi_t(\bs w)$ converges to the set $\mathcal H:= \mathcal E'\cap \{w_2={\frac{w_1}{w_1+w_4}}=\nicefrac12\}$, as claimed.
Moreover, the convergence is uniform on $\mathcal H_{\varepsilon} = \{\nicefrac12-\varepsilon\leq w_2\leq \nicefrac12+\varepsilon\}\cap \{\nicefrac12-\varepsilon\leq \frac{w_1}{w_1+w_4}\leq \nicefrac12+\varepsilon\}$ for $\varepsilon>0$ small enough, by a similar argument to Step (4) in Section~\ref{sec:cornet}.
Therefore, by Corollary~\ref{cor:pemantle}, $L(\hat{\bs W})\subset \mathcal H$ as claimed.

(2) For all $\bs w\in \mathcal H$,
\[F_3(\bs w) = \frac{w_3}{w_3+\nicefrac12} - w_3,\]
thus $F_3(\bs w) >0$, if $w_3<1/2$, and $F_3(\bs w)<0$ if $w_3>1/2$.
Thus, for all $\bs w\in \mathcal H$, $\Phi_t(\bs w)$ converges to $\mathcal H'$ as $t\to\infty$.
Furthermore, the convergence is uniform on $\mathcal H\cap\{\nicefrac12-\varepsilon\leq w_3\leq \nicefrac12+\varepsilon\}$ for $\varepsilon>0$ small enough, by similar argument to step (4) in Section~\ref{sec:cornet}, which concludes this step by Corollary~\ref{cor:pemantle}.

(3) For all $\bs w\in \mathcal H'$, 
\[F_1(\bs w) = 
\frac12+\frac12\cdot\frac{\frac{w_1}{1+w_1}\left(\frac12+\frac{\nicefrac12}{1+w_1}\right)}
{1-\frac12\frac{w_1}{1+w_1}-\frac{\nicefrac14}{(1+w_1)^2}}-w_1
=\frac12+\frac{w_1(2+w_1)}{2w_1^2+6w_1+3}-w_1. 
\]
Then one can check that $F_1(\bs w)>0$ if and only if
$f(w_1) >0$ where, for all $x\in \mathbb R$,
\[f(x)= -2x^3-4x^2+2x+\frac32.\]
Note that $f$ is a polynomial of degree~3, it thus has at most three zeros in $\mathbb R$.
One can check that $f'$ is positive on $((-2-\sqrt{7})/3,(-2+\sqrt{7})/3)$ and non-positive on the complement of this set. Thus, on $[0,1]$, $f$ is non-decreasing on $[0, (-2+\sqrt{7})/3]$ and non-increasing on $[(-2+\sqrt{7})/3, 1]$. Since $f(0)= \nicefrac32>0$ and $f(1)=-\nicefrac52$, we get that there exists a unique solution to $f(x) = 0$ on $[0,1]$, which we call $w^*$.
Moreover, $f(x)>0$ for all $x\in [0, w^*)$ and $f(x)<0$ for all $x\in (w^*, 1]$. 
The conclusion follows by Corollary~\ref{cor:pemantle} (to show assumption (iii), we use a similar argument to step (4) in Section~\ref{sec:cornet}).
\end{proof}

\appendix
\section{Calculating $F$ in the lozenge case: proof of~\eqref{eq:formule_lozenge}}
\label{sec:app}

We use the same notation as in Section~\ref{sec:lozenge}.
To prove~\eqref{eq:formule_lozenge}, we use the {electrical networks} method (see, e.g.~\cite{LL}).
We start by calculating $p_2(\bs w)$: this is the probability that a random walker on the graph with weights $\bs w = (w_i)_{1\leq i\leq 5}$, starting from $N$, crosses Edge~2 before crossing Edge~5.
This is equal to the probability that a walker starting from $N$ reaches $F_2$ before $F_5$ on the weighted graph of Figure~\ref{fig:p2}. 
We decompose $p_2(\bs w)$ according to the first step of the walker: 
\begin{equation}\label{eq:p2}
p_2(\bs w) = \frac{w_1}{w_1+w_4}\cdot p_{22}(\bs w)+\frac{w_4}{w_1+w_4}\cdot p_{25}(\bs w),
\end{equation}
where $p_{22}(\bs w)$ (resp.\ $p_{25}(\bs w)$) denotes the probability to reach $F_2$ before $F_5$ starting from $P_2$ (resp.\ $P_5$) on the graph of Figure~\ref{fig:p2}.
By classical formulas for random walks on weighted graphs (see, e.g.~\cite{LL}), 
\[p_{22}(\bs w) = \frac{\mathcal C_{P_2F_2}(\bs w)}{\mathcal C_{P_2F_2}(\bs w)+\mathcal C_{P_2F_5}(\bs w)},\]
where $\mathcal C_{XY}(\bs w)$ is the effective conductance between vertices $X$ and $Y$ when the edge weights are given by~$\bs w$. By definition of effective conductances, the effective conductance of a single edge is its weight. Thus, $\mathcal C_{P_2F_2} = w_2$.
Moreover, the conductance of two edges in parallel is the sum of their effective conductances, the effective conductance of two edges in series is the inverse of the sum of the inverses of their effective conductances.
\begin{figure}
\begin{center}
\includegraphics[width=15cm]{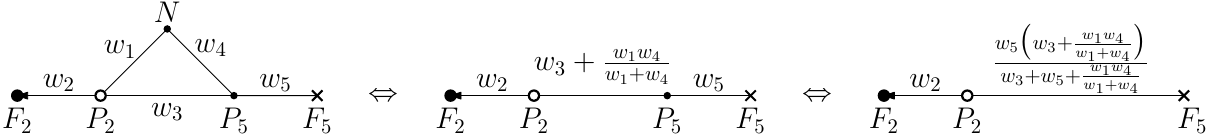}
\end{center}
\caption{Notation for the proof of \eqref{eq:formule_lozenge}, and calculation of $p_{22}(w)$, the probability of reaching $F_2$ before $F_5$ starting from $P_2$.}
\label{fig:p2}
\end{figure}
Using these formulas, we get (see Figure~\ref{fig:p2} for details)
\[\mathcal C_{P_2F_2}(\bs w) 
= \frac{\big(w_3+\frac{w_1w_4}{w_1+w_4}\big)w_5}{w_3+w_5+\frac{w_1w_4}{w_1+w_4}}.
\]
We thus get
\ba
p_{22}(\bs w)
&= \frac{w_2}{w_2 + \frac{\big(w_3+\frac{w_1w_4}{w_1+w_4}\big)w_5}{w_3+w_5+\frac{w_1w_4}{w_1+w_4}}}
= \frac{w_2\big(w_3+w_5+\frac{w_1w_4}{w_1+w_4}\big)}{w_2\big(w_3+w_5+\frac{w_1w_4}{w_1+w_4}\big) + \big(w_3+\frac{w_1w_4}{w_1+w_4}\big)w_5}
= \frac{w_2\big(w_3+w_5+\frac{w_1w_4}{w_1+w_4}\big)}{w_3+w_2w_5+\frac{w_1w_4}{w_1+w_4}},
\ea
because $w_2 +w_5=1$ for all $\bs w\in\mathcal E'$.
By symmetry,
\[p_{25}(\bs w)= 1- \frac{w_5\big(w_2+w_3+\frac{w_1w_4}{w_1+w_4}\big)}{w_3+w_2w_5+\frac{w_1w_4}{w_1+w_4}}
= \frac{(1-w_5)\big(w_3+\frac{w_1w_4}{w_1+w_4}\big)}{w_3+w_2w_5+\frac{w_1w_4}{w_1+w_4}}
= \frac{w_2\big(w_3+\frac{w_1w_4}{w_1+w_4}\big)}{w_3+w_2w_5+\frac{w_1w_4}{w_1+w_4}}.
\]
Thus,~\eqref{eq:p2} becomes
\ba
p_2(\bs w) 
&= \frac{w_1}{w_1+w_4}\cdot \frac{w_2\big(w_3+w_5+\frac{w_1w_4}{w_1+w_4}\big)}{w_3+w_2w_5+\frac{w_1w_4}{w_1+w_4}}
+\frac{w_4}{w_1+w_4}\cdot\frac{w_2\big(w_3+\frac{w_1w_4}{w_1+w_4}\big)}{w_3+w_2w_5+\frac{w_1w_4}{w_1+w_4}}\\
&= \frac{w_2(w_1(w_3+w_4+w_5)+w_3w_4)}{w_3+w_2w_5+\frac{w_1w_4}{w_1+w_4}},
\ea
as claimed.

\medskip
We now calculate $p_3(\bs w)$: we decompose on the first step of the random walker to get
\begin{equation}\label{eq:p3}
p_3(\bs w) 
= \frac{w_1}{w_1+w_4}\cdot p_{32}(\bs w) + \frac{w_4}{w_1+w_4}\cdot p_{35}(\bs w),
\end{equation}
where $p_{32}(\bs w)$ (resp.\ $p_{35}(\bs w)$) is the probability to cross edge 3 before reaching $F$ starting from $P_2$ (resp.\ $P_5$), when the edge weights are given by $\bs w$.
Decomposing over the first weight of a random walker starting at $P_2$, we get
\[p_{32}(\bs w) = \frac{w_1}{w_1+w_2+w_3}\cdot p_3(\bs w) + \frac{w_3}{w_1+w_2+w_3},\]
and similarly for $p_{35}(\bs w)$. 
Using this in~\eqref{eq:p3}, we get
\ba
&p_3(\bs w) \\
&= \frac{w_1}{w_1+w_4}\bigg(\frac{w_1}{w_1+w_2+w_3}\cdot p_3(\bs w) + \frac{w_3}{w_1+w_2+w_3}\bigg)
+\frac{w_1}{w_1+w_4}\bigg(\frac{w_4}{w_3+w_4+w_5}\cdot p_3(\bs w) + \frac{w_3}{w_3+w_4+w_5}\bigg),
\ea
which implies
\ba
&\bigg(1-\frac{w_1^2}{(w_1+w_4)(w_1+w_2+w_3)}-\frac{w_4^2}{(w_1+w_4)(w_3+w_4+w_5)}\bigg)p_3(\bs w)\\
&\hspace{8cm}= \frac{w_3}{w_1+w_4}\bigg(\frac{w_1}{w_1+w_2+w_3}+\frac{w_4}{w_3+w_4+w_5}\bigg).
\ea
This indeed gives the formula for $p_3(\bs w)$ announced in~\eqref{eq:formule_lozenge}.

\medskip
Finally, we show how to calculate $p_1(\bs w)$: again, we decompose according to the first step of the walker:
\begin{equation}\label{eq:p1_un}
p_1(\bs w) 
= \frac{w_1}{w_1+w_4} + \frac{w_4}{w_1+w_4}\cdot p_{15}(\bs w),
\end{equation}
where $p_{15}(\bs w)$ is the probability to cross edge~1 before reaching $F$ starting from $P_5$.
Decomposing according to the first step again, we get
\ban
p_{15}(\bs w)
&= \frac{w_4}{w_3+w_4+w_5}\cdot p_1(\bs w) + \frac{w_3}{w_3+w_4+w_5}\cdot p_{12}(\bs w)\notag\\
&= \frac{w_4}{w_3+w_4+w_5}\bigg( \frac{w_1}{w_1+w_4} + \frac{w_4}{w_1+w_4}\cdot p_{15}(\bs w)\bigg)+ \frac{w_3}{w_3+w_4+w_5}\cdot p_{12}(\bs w)\label{eq:p1_bis}
\ean
where $p_{12}(\bs w)$ is the probability to cross edge~1 before reaching $F$ starting from $P_2$.
We have used~\eqref{eq:p1_un} in the second equality.
Finally, we have
\begin{equation}\label{eq:p1_ter}
p_{12}(\bs w)= \frac{w_1}{w_1+w_2+w_3} + \frac{w_3}{w_1+w_2+w_3}\cdot p_{15}(\bs w).
\end{equation}
Using~\eqref{eq:p1_ter} in~\eqref{eq:p1_bis} we get
\[p_{15}(\bs w) 
= \frac{\frac{w_1}{w_1+w_4}\cdot\frac{w_4}{w_3+w_4+w_5}+ \frac{w_1}{w_1+w_2+w_3}\cdot\frac{w_3}{w_3+w_4+w_5}}
{1-\frac{w_4^2}{(w_1+w_4)(w_3+w_4+w_5)}-\frac{w_3^2}{(w_1+w_2+w_3)(w_3+w_4+w_5)}},\]
and we can then use in~\eqref{eq:p1_un} to get the formula for $p_1(\bs w)$ announced in~\eqref{eq:formule_lozenge}.

\bibliographystyle{alpha}
\bibliography{Fourmis}

\end{document}